\documentclass{article}   	
\usepackage{geometry}               
\usepackage{graphicx}
\usepackage[usenames,dvipsnames]{xcolor}
\usepackage{amssymb,amsmath,amsthm,amsfonts}
\usepackage{bm}                     
\usepackage[english]{babel}
\usepackage[latin1]{inputenc}
\usepackage[colorlinks]{hyperref}
\hypersetup{linkcolor=blue,citecolor=blue,filecolor=black,urlcolor=blue}
\usepackage[showonlyrefs]{mathtools}
\mathtoolsset{showonlyrefs=true}

\usepackage[sc]{mathpazo}

\usepackage{tikz}
\usepackage{pgfplots}

\pgfplotsset{compat=1.11}
\newcommand{\punto}{{node[circle, draw, fill=blue!50,
                        inner sep=0pt, minimum width=1.5pt]
{}}}

\newtheorem{proposition}{Proposition}[section]
\newtheorem{theorem}[proposition]{Theorem}
\newtheorem{corollary}[proposition]{Corollary}
\newtheorem{lemma}[proposition]{Lemma}

\theoremstyle{definition}

\theoremstyle{remark}
\newtheorem{remark}[proposition]{Remark}
\numberwithin{equation}{section}

\newcommand{\eps}{\varepsilon}

\newcommand{\R}{{\mathbb{R}}}

\newcommand{\Ecal}{{\mathcal{E}}}

\newcommand{\Gcal}{{\mathcal{G}}}

\newcommand{\cG}{{\mathcal G}}
\newcommand{\cF}{{\mathcal F}}

\newcommand{\cE}{{\mathcal E}}
\newcommand{\cA}{{\mathcal A}}
\newcommand{\cB}{{\mathcal B}}

\newcommand{\weak}{\rightharpoonup}

\newcommand{\pa}{\partial}

\newcommand{\En}{{E}}
\newcommand{\urr}{{u_\R}}
\newcommand{\uii}{{u_{[0,1]}}}

\title{Local minimizers in absence of ground states for the critical NLS energy on metric graphs}

\author{Dario Pierotti, Nicola Soave and Gianmaria Verzini}

\begin{document}
\maketitle

\begin{abstract}
We consider the mass-critical nonlinear Schr\"odinger equation on non-compact metric graphs. A quite complete description of the structure of the ground states, which correspond to global minimizers of the energy functional under a mass constraint, is provided by Adami, Serra and Tilli in \cite{AST2}, where it is proved that existence and properties of ground states depend in a crucial way on both the value of the mass, and the topological properties of the underlying graph. In this paper we address cases when ground states do not exist and show that, under suitable assumptions, constrained local minimizers of the energy do exist. This result paves the way to the existence of stable solutions in the time-dependent equation in cases where the ground state energy level is not achieved.
\end{abstract}
\noindent
{\footnotesize \textbf{AMS-Subject Classification}}.
{\footnotesize 35R02; 35Q55; 81Q35; 49J40.}

\section{Introduction}

Throughout this paper, we deal with non-compact connected metric graphs $\Gcal$, having a finite number of vertices and edges, where any edge $e$ is identified either with a closed bounded interval
$[0,|e|]$, or with (a copy of) the closed half-line $\R^+=[0,+\infty)$. On such a
$\mathcal{G}$, we consider the nonlinear Schr\"odinger (NLS) energy functional
\begin{equation}
\label{Encrit}
\En(u,\mathcal{G})=\frac{1}{2}\int_{\mathcal{G}}| u'|^2- \frac{1}{6}\int_{\mathcal{G}}| u|^{6},
\end{equation}
under the mass constraint
\begin{equation}
\label{masscon}
u\in H^1_{\mu}(\Gcal):=\left\{u\in H^1(\mathcal{G})\,,\,\int_{\mathcal{G}}| u|^2=\mu\right\},
\end{equation}
the Sobolev space  $H^1(\mathcal{G})$ consisting of all the continuous functions $u$ on $\Gcal$, such that
$\left.u\right|_{e} \in H^1(e)$ for every edge $e$.
As explained in \cite{AST2}, see also \cite{AST1,MR3494248, MR3758538}, the above energy is critical because, under mass-invariant dilations, the two terms in $\En$ scale in the same way.

Critical points of $\En(\cdot,\Gcal)$ constrained to $H^1_{\mu}(\Gcal)$ satisfy the NLS equation
\begin{equation}\label{eq:NLSE}
u'' + |u|^4 u = \lambda u
\end{equation}
on every edge, for a Lagrange multiplier $\lambda$ (which is the same on every edges, but may be different for different solutions); moreover, at each vertex the Kirchhoff condition is satisfied, which requires
 the sum of all the ingoing derivatives to vanish (see \cite[Prop. 3.3]{AST1}).
Through the usual \emph{ansatz} $\Phi(x,t) = e^{i\lambda t} u(x)$, such critical points correspond to solitary wave solutions to the Schr\"odinger equation
\[
i \partial_t \Phi (x,t) + \partial_{xx}\Phi(x,t) + |\Phi(x,t)|^4\Phi(x,t) = 0,
\qquad
x\in\Gcal, \ t>0,
\]
which appears in the Gross-Pitaevskii theory for Bose-Einstein condensation on graph-like structures (for more details on the physical interpretation, see \cite{AST2} and references therein). 

The study of Schr\"odinger equations on metric graphs has attracted considerable attention in the last decades. Linear problems have been extensively studied, see \cite{BK} and references therein. More recently, nonlinear problems have been addressed as well. We do not attempt to provide a complete overview of the many available results in the literature, for which we refer the interested reader to \cite{AST11, No} and the references therein. We limit to mention that results related to ours, concerning critical problems, have been obtained in \cite{AST2}, whose main results will be discussed in details in what follows; in \cite{Dnodea, DTcalcvar}, devoted to problems on periodic graphs and with localized nonlinearities, respectively; and, very recently, in \cite{NP}, regarding existence of standing waves (not necessarily ground states) of the NLS equation on the so-called tadpole graph. 

From the dynamical point of view, the more interesting critical points are the local minimizers, since they are natural candidates to correspond to orbitally stable solitary waves \cite{GrillakisShatahStrauss}.

The search of global minimizers, i.e. the ground state minimization problem
\begin{equation}
\label{eq:gslevel}
\Ecal_\Gcal(\mu):=\inf_{u\in H^1_{\mu}(\Gcal)}\En(u,\mathcal{G}),
\end{equation}
has been extensively investigated by Adami, Serra and Tilli in \cite{AST2}. As they describe, the
range of masses $\mu$ for which \eqref{eq:gslevel} is achieved strongly depends on the
topological properties of $\Gcal$. In case $\Gcal = \R$, seen as a pair of half-lines glued
together at the origin, $\Ecal_\R(\mu)$ is achieved if and only if $\mu$ is equal to the critical
mass $\mu_\R = \pi\sqrt 3/2$; actually, this is very well known, and similar results hold true
also in $\R^N$, $N\ge2$ (see for instance \cite{Cazenave2003}). Analogously, in case
$\Gcal = \R^+$, $\Ecal_{\R^+}(\mu)$ is achieved if and only if $\mu=\mu_{\R^+} = \mu_\R /2$. For
general $\Gcal$, several different situations may happen. On a general ground, every $\Gcal$ admits a critical mass $\mu_\Gcal$, where
\[
\mu_{\R^+}\le\mu_\Gcal\le\mu_{\R}
\]
(see \cite[Prop. 2.3]{AST2}), and a necessary condition for $\Ecal_\Gcal(\mu)$ to be achieved is that
\[
\mu_\Gcal\le\mu\le\mu_{\R}
\]
(see \cite[Coro. 2.5]{AST2}). Such condition is far from being sufficient, and
\cite{AST2} provides a classification of different kinds of graph, for which the set of masses
$\mu$ allowing for a ground state can be either an interval, reduce to a point, or even be empty. We
postpone a more detailed discussion of the results in \cite{AST2} below.

The main aim of this paper is to show that, under fairly general assumptions, there may exist
local minimizers of $E$ in $H^1_\mu(\Gcal)$, for values of $\mu$ strictly smaller than
$\mu_\Gcal$, i.e. in cases when a global minimizer can not exist. To state our main result, we
assume w.l.o.g. that any vertex of $\Gcal$ has degree different from $2$ (this is possible whenever $\Gcal$ is not isometric to $\R$, which can be decomposed in two half-lines only allowing a vertex of degree $2$). With some abuse of
notation, we call a ``open half-line $\ell\subset\Gcal$'' any unbounded edge of the underlying combinatorial graph, and we
denote with $\Gcal\setminus\ell$ the graph obtained from $\Gcal$ by removing the edge $\ell$ and the corresponding vertex at infinity.

Our main result is the following.
\begin{theorem}\label{thm:main_intro}
Let $\Gcal$ be a non-compact connected metric graph, having a finite number of vertices and edges. Let us assume that
\begin{enumerate}
 \item $\Gcal$ has at least two half-lines and at least one bounded edge;
 \item for every open half-line $\ell\subset\Gcal$, $\mu_{\Gcal\setminus\ell} = \mu_\Gcal$.
\end{enumerate}
Then there exists $\bar \mu \in (0, \mu_{\Gcal})$ such that for every $\mu \in (\bar \mu, \mu_{\Gcal})$ the functional $\En(\cdot, \cG)$ has a critical point on $H^1_{\mu}(\cG)$, which is a local minimizer.
\end{theorem}
Exploiting the results in \cite{AST2}, we can provide an explicit characterization of different
classes of graphs fulfilling our assumptions, see below. As an example, Theorem
\ref{thm:main_intro} applies to any $\Gcal$ having at least two half-lines and one terminal
point (a tip), the simplest prototype being the one illustrated in Fig. \ref{fig:tip}. Notably, this kind of graphs admits no
global minimizers, regardless of the choice of $\mu$.
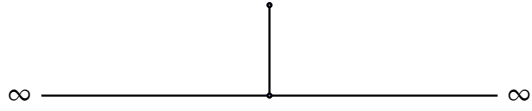
\begin{figure}
\begin{center}
\begin{tikzpicture}[thick]%
    \draw
 (0,0) node[left]{$\infty$} -- 
 (3,0) \punto --
 (6,0) node[right]{$\infty$}
 (3,0) -- (3,1.2) \punto
;
\end{tikzpicture}
\caption{graph with a tip and two half-lines.\label{fig:tip}}
\end{center}
\end{figure}

We remark that normalized local minimizers of NLS-type energies, on standard domains,
have been recently found in different contexts: we refer to \cite{ntvAnPDE,ntvDCDS,MR3689156,MR3918087} for NLS equations and systems on bounded domains of
$\R^{N}$, to \cite{BellazziniJeanjean,MR3638314} for problems on $\R^N$ with potentials, to \cite{BellazziniGeorgievVisciglia} for a semi-relativistic case, and to \cite{Nic1, Nic2} for equations with combined nonlinearities.

To better illustrate our result, let us provide more details about the results contained in
\cite{AST2}, about the global minimization problem \eqref{eq:gslevel}. The authors there detect
four mutually exclusive cases \cite[Thms. 3.1--3.4]{AST2}.
\begin{enumerate}
 \item \emph{$\Gcal$ has at least a terminal point.} A terminal point (a tip) is a vertex, not at infinity, of degree 1. In this case $\mu_\Gcal = \mu_{\R^+}$, and \eqref{eq:gslevel} is never achieved unless $\Gcal = \R^+$ and $\mu = \mu_{\R^+}$;
 \item \emph{$\Gcal$ admits a cycle covering.} Here ``cycle'' means either a bounded loop, or an unbounded path joining two distinct points at infinity. Equivalently, $\Gcal$ has at least two half-lines and no terminal point, and whenever $\Gcal\setminus e$ has two connected components, both are unbounded (here $e$ denotes any bounded edge). Then $\mu_\Gcal = \mu_{\R}$, and \eqref{eq:gslevel} is never achieved unless $\Gcal = \R$ or $\Gcal$ is a ``bubble tower'' (see \cite{AST2}), and $ \mu= \mu_{\R}$;
 \item \emph{$\Gcal$ has exactly one half-line and no terminal point.} Then $\mu_\Gcal =
 \mu_{\R^+}$  and \eqref{eq:gslevel} is achieved if and only if $\mu\in (\mu_{\R^+},\mu_{\R}]$;
 \item \emph{$\Gcal$ does not belong to any of the previous three cases, i.e. it has no tips, no cycle-covering and at least 2 half-lines.} Then, in case
 $\mu_\Gcal < \mu_{\R}$  we have that \eqref{eq:gslevel}
 is achieved if and only if $\mu\in[\mu_{\Gcal},\mu_{\R}]$; if $\mu_\Gcal = \mu_{\R}$, then nothing is known.
\end{enumerate}
Using such classification, we easily see that the following types of graph $\Gcal$ fulfill the assumptions of Theorem \ref{thm:main_intro}.
\begin{itemize}
 \item \emph{$\Gcal$ has at least a terminal point and at least two half-lines (Fig. \ref{fig:tip}).} Indeed, both
$\Gcal$ and $\Gcal\setminus\ell$, for any open half-line $\ell\subset\Gcal$, fall into case 1. above, and $\mu_{\Gcal\setminus\ell} = \mu_\Gcal = \mu_{\R^+}$.
 \item \emph{$\Gcal$ has at least a bounded edge and both $\Gcal$ and any $\Gcal\setminus\ell$ admit a cycle covering (Fig. \ref{fig:cycle1}, left). }%
\begin{figure}
\begin{center}
\begin{tikzpicture}[thick]%
    \draw
 (2,0) node[left]{$\infty$} -- 
 (3,0) \punto
 (2,2) node[left]{$\infty$} -- 
 (3,2) \punto
 (3,0) to[out=45,in=-140] (3,2) -- (4,1.7) \punto -- (4.2,.5) \punto -- cycle
 (3,0) -- (4.6,0) \punto -- (4.2,.5)
 (4.6,0) -- (5,0)
 (5,0) \punto -- (6.2,0) \punto -- (6,2) \punto -- cycle
 (6.2,0) -- 
 (7,0) node[right]{$\infty$}
 (6,2) -- 
 (7,2) node[right]{$\infty$}
;
\end{tikzpicture}
\qquad
\begin{tikzpicture}[thick]%
    \draw
 (1.3,0) node[left]{$\infty$} -- 
 (3,0) \punto --
 (3,.4) \punto arc (-90:0:.8) \punto arc (0:270:.8)
 (3,0) -- 
 (5.3,0) node[right]{$\infty$}
 (3.8,1.2) -- node[above, pos=.55]{$\ell^*$}
 (5.3,1.2) node[right]{$\infty$}
;
\end{tikzpicture}
\caption{the graph on the left fulfills the assumptions of Theorem \ref{thm:main_intro}, while that on the right does not.\label{fig:cycle1}}
\end{center}
\end{figure}
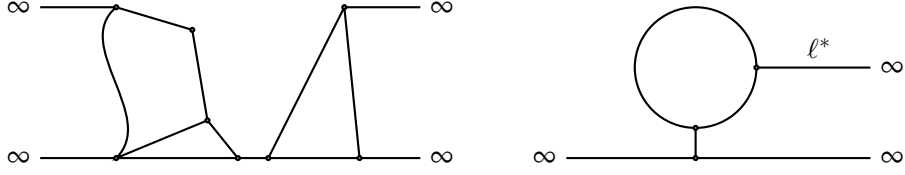%
Then, both
$\Gcal$ and any $\Gcal\setminus\ell$ fall into case 2. above, with $\mu_{\Gcal\setminus\ell} = \mu_\Gcal = \mu_{\R}$. Sufficient conditions for this to hold are that $\cG$ admits a cycle covering, has at least three half-lines and a bounded edge, and no cutting edge. Such conditions are not necessary, as the figure shows. On the other hand, graphs with a cycle covering and at least three half-lines may not satisfy the assumptions of Theorem \ref{thm:main_intro}: for instance,
removing $\ell^*$ from the graph in Fig. \ref{fig:cycle1}, right, we obtain the so-called ``signpost'' graph, which falls into case 4. above and whose critical mass is known to be strictly smaller than $\mu_\R$ (see \cite{AST2}). Similarly, our result does not apply to star-graphs, which do not have bounded edges (see Remark \ref{rem:stargraph} for a deeper discussion of this case).
\end{itemize}
We observe also that Theorem \ref{thm:main_intro} does not apply to graphs satisfying
3. above, since they have exactly one half-line (a prototypical example of this class of graphs is the tadpole graph, studied in \cite{NP}). Finally, to clarify whether Theorem \ref{thm:main_intro} applies to graphs of type 4., we provide
the following result.
\begin{proposition}\label{prop:removing}
Let $\Gcal$ be a non-compact metric graph, not isometric to $\R^+$, having a finite number of vertices and edges.
If $\Ecal_\Gcal(\mu_{\Gcal})$ is achieved then there exists an open half-line
$\ell\subset\Gcal$ such that $\mu_{\Gcal\setminus\ell} < \mu_\Gcal$.
\end{proposition}
Incidentally, the above proposition is of independent interest, as it helps in supplementing
\cite[Thm. 3.4]{AST2}.
\begin{corollary}\label{coro:supplementAST}
If $\Gcal$ falls in case 4. above (i.e. it has no tips, no cycle-covering and at least 2 half-lines) then $\mu_\Gcal>\mu_{\R^+}$.
\end{corollary}
Based on Proposition \ref{prop:removing} we see that, in case $\Gcal$ belongs to case 4.
above and $\mu_\Gcal<\mu_\R$, then Theorem \ref{thm:main_intro} does not apply. On the contrary,
it may apply to some $\Gcal$ in case 4. not achieving $\mu_\Gcal = \mu_\R$, even though we are
not aware of any example of such a graph.

To conclude this introduction, we want to describe the strategy we followed in order to guess
that a result like Theorem \ref{thm:main_intro} may hold. Indeed, this strategy is quite
elementary, and we think that it may be applied successfully also to other cases. It is based on three steps.

\textbf{Step 1: phase plane analysis of a model graph.} Normalized solutions, i.e. critical
points of $E(\cdot, \Gcal)$ in $H^1_\mu(\Gcal)$, with assigned $\mu$, can not be find by
elementary methods, even when $\Gcal$ is a fixed graph with simple structure. This is not the
case if, instead of assigning $\mu$, we use as a parameter the Lagrange multiplier $\lambda$ in
\eqref{eq:NLSE}. For concreteness, we consider the prototype
graph $\Gcal$ made by two half-lines and a segment, depicted as in Fig. \ref{fig:tip}. For any
fixed $\lambda$ we look for solutions of the NLSE \eqref{eq:NLSE} on every edge of
$\Gcal$, complemented with Kirchhoff conditions at the vertex. This can be easily done via
elementary phase plane analysis in different ways, obtaining families of solutions
\[
\lambda \mapsto u_\lambda \in H^1_{\mu(\lambda)}(\Gcal),
\qquad
\text{where }\mu(\lambda) := \int_\Gcal u_\lambda^2.
\]
In this way $u_\lambda$ is a critical point of $E(\cdot, \Gcal)$ in $H^1_{\mu(\lambda)}(\Gcal)$,
and both $u_\lambda$ and $\mu(\lambda)$ are either explicit or easy to be numerically estimated.

\textbf{Step 2: detection of candidate local minimizers.} Once an explicit family
$\lambda \mapsto u_\lambda$ is constructed as in Step 1, we would like to detect if it consists
in local minimizers. A very powerful tool to this aim is the celebrated stability theory
developed by Grillakis, Shatah and Strauss \cite{GrillakisShatahStrauss,MR1081647}. Roughly, in
our context such theory implies that $u_\lambda$, $0<a<\lambda<b$, is a local minimizer of $E(\cdot, \Gcal)$ in $H^1_{\mu(\lambda)}(\Gcal)$ if
\begin{itemize}
 \item $u_\lambda$, as a critical point of the action functional $A_\lambda(u,\Gcal) =
  E(u, \Gcal) + \frac{\lambda}{2}\int_\Gcal u^2$, is non-degenerate, it has Morse index $1$, and
 \item the map $\lambda\mapsto\mu(\lambda)$ is increasing.
\end{itemize}

\textbf{Step 3: identification of the variational structure and extension to more general
graphs.} Contrarily to saddle points, local minimizers are structurally stable. In the model case
$u_\lambda$ is explicit. Once a neighborhood of $u_\lambda$ is identified, in which $u_\lambda$
is a global minimizer, we can try to spot similar neighborhoods on more general graphs. Of
course, this is the more delicate part of the strategy, and will involve the minimization of the energy under a double constraint, which will induce a convenient splitting of the mass on the compact and non-compact parts of the graph. Different types of doubly constrained variational problems for the NLS energy were considered also in \cite{AST10, NP}.

The paper is structured as follows. In Section \ref{sec:prel} we prove Proposition \ref{prop:removing} and Corollary \ref{coro:supplementAST}. In Section
\ref{soltip} we develop the phase plane analysis of the model graph illustrated in Fig. \ref{fig:tip} (Step 1 of the above strategy). Section \ref{sec: cpt} is devoted to a general compactness argument for locally minimizing sequences, which is then applied in Section \ref{sec:proofmain} to prove Theorem
\ref{thm:main_intro}. Finally, in Appendix \ref{app:A} we sketch the application of the Grillakis-Shatah-Strauss theory to the explicit solutions on the model graph (Step 2 of the above strategy).

\section{Proof of Proposition \ref{prop:removing}}\label{sec:prel}

With the same notations as in \cite{AST2}, we recall that on a general non-compact metric graph $\cG$ the Gagliardo-Nirenberg inequality
\[
\|u\|_{L^6(\cG)}^6 \le \frac{3}{\mu_{\cG}^2} \|u\|_{L^2(\cG)}^4 \|u'\|_{L^2(\cG)}^2 \qquad \forall u \in H^1(\cG)
\]
holds, where $\mu_{\cG}$ denotes the critical mass of $\cG$. It follows plainly that
\begin{equation}\label{coerc1}
\En(u, \cG) \ge \frac12 \left( 1- \left(\frac{\mu}{\mu_{\cG}}\right)^2\right)\|u'\|_{L^2(\cG)}^2 \qquad \forall u \in H^1_\mu(\cG).
\end{equation}
In particular, for $\mu\le \mu_{\mathcal{G}}$ the functional $\En(\cdot,\mathcal{G})$ is non-negative. Moreover, from \cite[Prop. 2.4]{AST2} we know that
\begin{equation}\label{eq:gsmu}
\begin{aligned}
\Ecal_\Gcal(\mu) &=0  &&\text{if }\mu \le \mu_\Gcal, \\
\Ecal_\Gcal(\mu) &<0\text{ (possibly $-\infty$)} && \text{if }\mu > \mu_\Gcal.
\end{aligned}
\end{equation}
\begin{proof}[Proof of Proposition \ref{prop:removing}]
Since $\Gcal$ is not isometric to $\R^+$, and $\Ecal_\Gcal(\mu_\Gcal)$ is achieved, we deduce by the classification provided in \cite{AST2} that $\Gcal$ has
at least two half-lines, say $\ell_1$ and $\ell_2$. Notice that, in case $\Gcal$ is isometric to $\R$, the proposition is trivial: indeed, in such case $\Gcal\setminus\ell\equiv\R^+$, and $\mu_{\R^+} = \mu_\R/2<\mu_\R$. As a consequence, we are left to treat the case
in which $\Gcal\setminus(\ell_1\cup\ell_2)$ has positive measure (possibly infinite, in case $\Gcal$ has at least three half-lines). Let $\bar u\in H^1_{\mu_{\Gcal}}(\Gcal)$, strictly positive on $\Gcal$,
be such that $E(\bar u,\Gcal)=\Ecal_\Gcal(\mu_{\Gcal})=0$. We have
\[
\int_{\ell_1} \bar u^2  + \int_{\ell_2} \bar u^2  < \int_{\Gcal} \bar u^2 = \mu_\Gcal \le \mu_\R.
\]
We deduce that at least one half-line, say $\ell_1$, satisfies
\[
\int_{\ell_1} \bar u^2=:\eta < \frac{\mu_\R}{2}=\mu_{\R^+},
\]
therefore, by \eqref{coerc1},
\[
E(\left.\bar u\right|_{\ell_1},\ell_1)>0.
\]
Then we obtain
\[
\Ecal_{\Gcal\setminus{\ell_1}}(\mu_\Gcal-\eta) \le E(\left.\bar u\right|_{\Gcal\setminus{\ell_1}},\Gcal\setminus{\ell_1}) =
E(\bar u,\Gcal) - E(\left.\bar u\right|_{\ell_1},\ell_1)<0.
\]
Recalling \eqref{eq:gsmu}, this forces $\mu_{\Gcal\setminus\ell_1} \le \mu_\Gcal-\eta$, and the proposition follows.
\end{proof}
\begin{proof}[Proof of Corollary \ref{coro:supplementAST}]
Assume by contradiction that $\Gcal$ has no tips, no cycle-covering and at least 2 half-lines, and $\mu_\Gcal=\mu_{\R^+}$. Then, by
\cite[Thm. 3.4]{AST2}, $\mu_\Gcal$ is achieved. But then Proposition \ref{prop:removing} implies that, for some $\ell$,
the non-compact graph $\Gcal\setminus\ell$ satisfies $\mu_{\Gcal\setminus\ell} < \mu_{\R^+}$, in contradiction with
\cite[Prop. 2.3]{AST2}.
\end{proof}

\section{Direct analysis of the model case.}
\label{soltip}

Let $\mathcal{G}$ be a metric graph consisting of a straight line, identified with $\R$, with one pendant attached at the origin (see Fig. \ref{fig:tip}). Without loss of generality, we identify the pendant with the interval $[0,1]$. As we mentioned, it was proved in \cite[Thm. 3.1]{AST2} that the ground-state energy level \eqref{eq:gslevel} is never
attained. In this section we show that there exist positive constrained critical levels of $\En(u,\mathcal{G})$ for suitable intervals of $\mu$, some of which correspond to candidate local minimizers.

We recall that $u\in H^1_{\mu}(\Gcal)$ is a constrained critical point of $\En$ if and only if it solves, on every edge of $\mathcal{G}$, the stationary NLS equation
\begin{equation}
\label{equcrit}
u''+|u|^4u=\frac{\Lambda^2}{3}u\,,
\end{equation}
for some value of $\Lambda$ not depending on the edge, and the Kirchhoff condition holds at any vertex of $\Gcal$ (see \cite[Prop. 3.3]{AST1}). For every $\Lambda>0$, the equation \eqref{equcrit}  is solved on $\mathbb{R}$ by the family of solitons
\begin{equation}
\label{defsol}
\phi_{\Lambda}(x)=\sqrt{\Lambda}\,\phi(\Lambda x)\,,\quad\quad\quad \phi(x)=\cosh^{-1/2}(2x/\sqrt 3)\,,
\end{equation}
\emph{all of mass} $\mu_{\mathbb{R}}$. We will define a solution of \eqref{equcrit} made of two symmetric pieces of translated solitons in $\mathbb{R}$ and of a suitable solution defined in $[0,1]$. As a matter of fact, we mimic an analogous construction performed in \cite[Sec. 6]{AST1} to provide a ground state in the subcritical case.

Let us define, for every $y>0$ fixed,
\begin{equation}
\urr(x) := u\big |_{\mathbb{R}}(x)=\left\{
                                                \begin{array}{ll}
                                                    \phi_{\Lambda}(x+y), & \hbox{if $x\ge 0$;} \\
                                                    \phi_{\Lambda}(x-y), & \hbox{if $x<0$.}
                                                  \end{array}
                                                \right.
\end{equation}
Note that $\urr$ is continuous on $\mathbb{R}$ since
\begin{equation}
\label{philam}
\phi_{\Lambda}(y)=\phi_{\Lambda}(-y)=\sqrt{\Lambda}\cosh^{-1/2}(2\Lambda y/\sqrt{3})\,.
\end{equation}
On the other hand, there is a jump of the derivatives at the origin as
$$
\phi'_{\Lambda}(\pm y) =
\mp\frac{\Lambda^{3/2}}{\sqrt{3}}\frac{\sinh(2\Lambda y/\sqrt{3})}{\cosh^{3/2}(2\Lambda y/\sqrt{3})}.
$$
On the interval $[0,1]$ we define $\uii:=u\big |_{[0,1]}$ by solving \eqref{equcrit}  with overdetermined data assigned by the
continuity at $x=0$ and Kirchhoff conditions:
\begin{equation}
\label{initcond}
\uii(0)=\phi_{\Lambda}(y), \qquad
\uii'(0)=2|\phi'_{\Lambda}(y)|,
\end{equation}
\begin{equation}
\label{kircon}
\uii'(1)=0.
\end{equation}
Note that, by elementary calculations, one can write the above term in the form
\begin{equation}
\label{initcond2b}
2|\phi'_{\Lambda}(y)|=\frac{2}{\sqrt{3}}\phi_{\Lambda}(y) \big (\Lambda^2-\phi_{\Lambda}(y)^4\big )^{1/2}.
\end{equation}
The solutions of equation \eqref{equcrit} are conveniently represented by the orbits of the associated Hamiltonian system in the $(u,u')$-plane, see Fig. \ref{fig:pp}; in fact, any solution $u$ on a connected interval satisfies
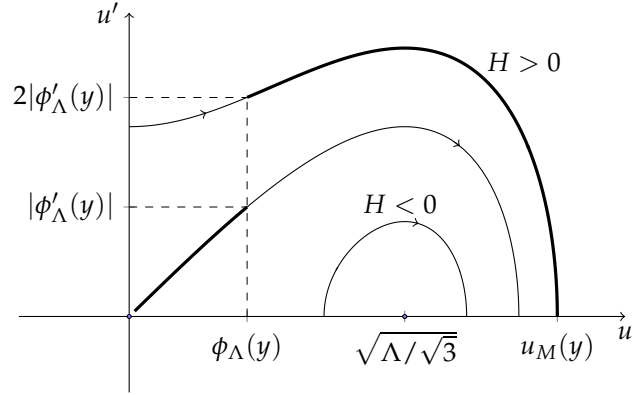
\begin{figure}[htbp]
\begin{center}
\begin{tikzpicture}
\begin{axis}[
   axis lines=center,
   width=.65\linewidth,
   height=.45\linewidth,
   axis line style={->},
   typeset ticklabels with strut,
   x label style={below},
   y label style={left},
   xmin=-.4, xmax=1.8,
   ymin=-.2, ymax=.8,
   xlabel={$u$},
   ylabel={$u'$},
   xtick={.428,1,1.5538},
   xticklabels={$\phi_\Lambda(y)$,$\sqrt{\Lambda/\sqrt{3}}$,$u_M(y)$},
   ytick={0.2884,0.5769},
   yticklabels={$|\phi'_\Lambda(y)|$,$2|\phi'_\Lambda(y)|$},
   ]
\draw (1.6,.72) node[below left, fill=white] {$H>0$};
\draw (1.6,.72) node[below left, fill=white] {$H>0$};
\draw (1.15,.345) node[below left, fill=white] {$H<0$};
\draw (0,0) \punto (1,0) \punto;
\addplot [domain=0.02:.428, very thick]{.5*sqrt(2*x^2-x^4))};
\addplot [domain=.428:1.2, samples=80, ->]{.5*sqrt(2*x^2-x^4))};
\addplot [domain=0:.46]({sqrt(1+sqrt(1-4*x^2))},{x});
\addplot [domain=0:.28, ->]{.5*sqrt(1+2*x^2-x^4))};
\addplot [domain=0.27:.428]{.5*sqrt(1+2*x^2-x^4))};
\addplot [domain=.428:1.45, samples=80, very thick]{.5*sqrt(1+2*x^2-x^4))};
\addplot [domain=0:.45, very thick]({sqrt(1+sqrt(2-4*x^2))},{x});
\addplot [domain=.75:1.05, samples=80, ->]{.5*sqrt(-.75+2*x^2-x^4))};
\addplot [domain=1.04:1.2, samples=80]{.5*sqrt(-.75+2*x^2-x^4))};
\addplot [domain=0:.12]({sqrt(1+sqrt(.25-4*x^2))},{x});
\addplot [domain=0:.13]({sqrt(1-sqrt(.25-4*x^2))},{x});
\addplot [domain=0:0.5769, very thin, dashed] ({0.428},{x});
\addplot [domain=0:.428, very thin, dashed] {0.2884};
\addplot [domain=0:.428, very thin, dashed] {0.5769};
\end{axis}
\end{tikzpicture}
\caption{phase plane analysis for the model graph: the thick part on the homoclinic trajectory corresponds to the solution on each half-line $\R_\pm$; the thick part on the trajectory with $H>0$ corresponds to the solution on the pendant $[0,1]$.
\label{fig:pp}}
\end{center}
\end{figure}
\begin{equation}
\label{hamsys}
\frac{1}{2}|u'|^2+\frac{1}{6}|u|^6-\frac{\Lambda^2}{6}|u|^2=H\,,
\end{equation}
for some $H\ge H_m = -\Lambda^3/9\sqrt 3$ (the minimum of the potential energy). The system has three equilibrium solutions $u=0$ (with $H=0$) and $u=\pm\sqrt{\Lambda/\sqrt{3}}$ (with $H=H_m<0$). Moreover, the level set $H=0$ also contains the homoclinic orbit corresponding to the soliton solution \eqref{defsol} and its translates; such orbit intersects the $u$ axis also at $u=\sqrt{\Lambda}$. By reflecting with respect to the $u'$ axis we have the homoclinic associated to minus the soliton solution.
Finally, any other  level curve with $H\neq 0$ (both positive and negative) is a closed orbit corresponding to\emph{ periodic solutions} of the system; if $H<0$, the orbit is internal to one homoclinic, while it circles both the homoclinics for $H>0$.

Now, to every point $(\phi_{\Lambda}(y),\phi'_{\Lambda}(y))$ on the (right) homoclinic we associate the `initial point'
$(\phi_{\Lambda}(y),2|\phi'_{\Lambda}(y)|)$ on the level curve $H=H(y)$, where, by \eqref{initcond2b} and \eqref{hamsys},
\begin{equation}
\label{Ey}
H(y):= \frac{2}{3}\phi_{\Lambda}(y)^2 \big (\Lambda^2-\phi_{\Lambda}(y)^4\big )-\frac{\Lambda^2}{6}\phi_{\Lambda}(y)^2
+\frac{1}{6}\phi_{\Lambda}(y)^6=\frac{1}{2}\phi_{\Lambda}(y)^2 \big (\Lambda^2-\phi_{\Lambda}(y)^4\big )>0\,.
\end{equation}
By looking at the orbit with $H=H(y)$, one sees that it intersects the $u$ axis at the points $\pm u_M$, where
$$u_M=u_M(y)>\sqrt{\Lambda}$$
is the positive solution of
\begin{equation}
\label{uMy}
u_M^6 -\Lambda^2 u_M^2=6\, H(y)\,.
\end{equation}
We claim that, by suitably choosing $\Lambda$, we can provide a solution $u$ of \eqref{hamsys} on $(0,1)$ (with $H=H(y)$), which satisfies
\eqref{initcond} and such that $u(1)=u_M$; thus, at the end of the pendant the solution $u$ will satisfy \eqref{kircon} as well.

In order to prove the above claim, we first note that, by \eqref{philam}, \eqref{Ey}, \eqref{uMy}, the quantities
\begin{equation}
\label{stars}
\phi^*=\phi_{\Lambda}(y)/\sqrt{\Lambda},\quad  H^*=H(y)/\Lambda^3,\quad u_M^*=u_M/\sqrt{\Lambda},\,
\end{equation}
only depend on the product $z=\Lambda y$. Moreover, we have $0<\phi^*<1$, $ u_M^*>1$, and
\begin{equation}
\label{estar}
6\,H^*={u_M^*}^6-{u_M^*}^2\,.
\end{equation}
Solving \eqref{hamsys} with respect to $u'$ and choosing $du/dx\ge 0$, we have
\begin{equation}
x(u_M)=\sqrt 3\int_{\phi_{\Lambda}}^{u_M}\frac{du}{\sqrt{6H(y)-u^6+\Lambda^2 u^2}}
\end{equation}
By the substitution $v=u/\sqrt{\Lambda}$ we get
\begin{equation}
x(u_M)=\frac{\sqrt 3}{\Lambda}\int_{\phi^*}^{u_M^*}\frac{dv}{\sqrt{6H^*-v^6+v^2}}.
\end{equation}
The integral above is a positive function of $z$; hence, for every fixed $z>0$ we can take

\begin{equation}
\label{lamz}
\Lambda=\Lambda(z)=\sqrt 3\int_{\phi^*}^{u_M^*}\frac{dv}{\sqrt{6H^*-v^6+v^2}}.
\end{equation}
so that $x(u_M)=1$.
\begin{remark}
\label{lambn}
The above value of $\Lambda$ is not the unique choice satisfying \eqref{kircon}; for example we could ``add a half rotation'' around the origin in the $(u,u')$-plane, requiring that
$u=-u_M$ at $x=1$. By defining
$$\frac{T_{\Lambda}}{2}=\sqrt 3\int_{-u_M}^{u_M}\frac{du}{\sqrt{6H(y)-u^6+\Lambda^2 u^2}}=\frac{T^*}{2{\Lambda}}\,$$
we obtain the condition
$$
\frac{\sqrt 3}{\Lambda}\int_{\phi^*}^{u_M^*}\frac{dv}{\sqrt{6H^*-v^6+v^2}}+\frac{T^*}{2{\Lambda}}=1.
$$
More generally, we find (for every fixed $z>0$) a sequence of admissible values
\begin{equation}
\label{admlam}
\Lambda_n=\sqrt 3\int_{\phi^*}^{u_M^*}\frac{dv}{\sqrt{6H^*-v^6+v^2}}+n\,\frac{T^*}{2}\,,\quad n=0,1,2,...
\end{equation}
and a corresponding sequence of solutions $u_n$. Due to the further oscillations, it is possible to prove that these solutions have higher Morse index. For this reason, in the rest of the section we will focus on the case $n=0$, i.e. on the choice of $\Lambda$ provided in \eqref{lamz}.
\end{remark}
Going back to the choice \eqref{lamz}, we have that the total mass of the corresponding solution $u=u_z$ only depends on $z$.
Indeed, the mass on the pendant is given by
\begin{equation}\label{masstip}
\int_0^1 \uii(x)^2\,dx=\sqrt 3\int_{\phi_{\Lambda}}^{u_M}\frac{u^2}{\sqrt{6H(y)-u^6+\Lambda^2 u^2}}\,du
=\sqrt{3} \int_{\phi^*}^{u_M^*}\frac{v^2}{\sqrt{6H^*-v^6+v^2}}\, dv,
\end{equation}
while the mass on $\mathbb{R}$ is simply
\begin{equation}
\label{massr}
\int_\R \urr(x)^2\,dx= \int_{|x|>y} \frac{\Lambda}{\cosh(2\Lambda x/\sqrt 3)}\,dx=
\int_{|x'|>z} \frac{1}{\cosh(2 x'/\sqrt 3)}\,dx'\,,
\end{equation}
and all the above integrals are (continuous) functions of $z$ alone. Thus, as $z$ spans $(0,+\infty)$, we have an interval, or \emph{allowed band}, for the values of the total mass
\begin{equation}
\label{lowmass}
\mu(z):=\int_{|x|>z} \frac{1}{\cosh(2 x/\sqrt 3)}\,dx+\sqrt{3} \int_{\phi^*}^{u_M^*}\frac{v^2}{\sqrt{6H^*-v^6+v^2}}\, dv.
\end{equation}
A rigorous study of the map $z\mapsto\mu(z)$ is contained in Appendix \ref{app:A2}. Anyway, a qualitative picture of its behavior can be easily obtained by a numerical evaluation of the second integral above (the first one is explicit), see Fig. \ref{fig:num}.
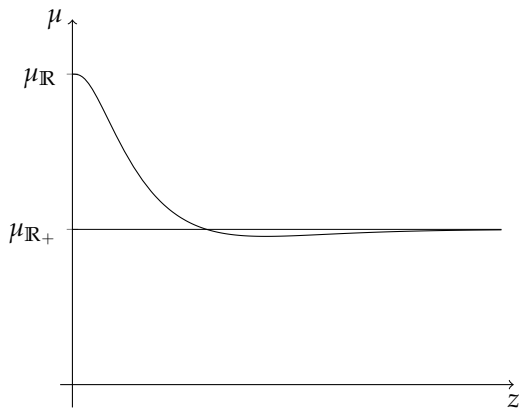
\begin{figure}[htbp]
\begin{center}
\begin{tikzpicture}
    \begin{axis}[
   scale only axis,
   axis lines=center,
   height=.35\linewidth,
   axis line style={->},
   typeset ticklabels with strut,
   x label style={below},
   y label style={left},
   xmin=-.2, xmax=7.2,
   ymin=-.2, ymax=3.2,
   xlabel={$z$},
   ylabel={$\mu$},
   xtick={\empty},
   xticklabels={$\frac{N+2}{NC_N}$},
   ytick={2.7207,1.36035},
   yticklabels={$\mu_\R$,$\mu_{\R_+}$},
   ]
            \addplot[
    black,
    mark size=0.1pt]
    table {plot.dat};
    \draw[very thin] (0,1.36035) -- (7,1.36035);
     \end{axis}
\end{tikzpicture}
\caption{numerical behavior of the map $z\mapsto \mu(z)$.}
\label{fig:num}
\end{center}
\end{figure}
Roughly, we can say that when $z\to 0^+$ we have $\phi^*\to 1^-$ and still $u_M^*\to 1^+$, so that the total mass converges to the soliton mass $\mu_\mathbb{R}$ (on the real line). On the other hand, in the limit $z\to +\infty$, an half-soliton tends to concentrate on the pendant (it can be shown by \eqref{lamz}, \eqref{lowmass} that $\Lambda(z)\to +\infty$ and $\mu(z)\to \mu_{\R^+}$ in this limit); furthermore, as we show in  Appendix \ref{app:A2}, the limit of $\mu(z)$ for $z\to +\infty$ \emph{is approached from below}, so that this function will have a minimum $\mu_m$ in $(0,+\infty)$ whose value must be \emph{strictly lower than} $\mu_\mathbb{R^+}$. This has two consequences: firstly, it follows that the band of allowed masses for the critical points of \eqref{Encrit} includes the interval
$[\mu_m,\mu_\mathbb{R})$:
$$
\mu(\mathbb{R}^+) = [\mu_m,\mu_\mathbb{R})\,.
$$
Secondly, by the theory of Grillakis-Shatah-Strauss, we can show that the corresponding solutions are orbitally stable for $z$ large (in particular, for every
$\mu\in (\mu_m,\mu_\mathbb{R^+})$ there are at least two critical points of \eqref{Encrit} in $H^1_{\mu}$, one local minimizer and one of mountain pass type).

This suggests to look for local minimizers of the energy, for masses slightly smaller than $\mu_\mathbb{R^+}$, also for more general graphs with a tip.
\begin{remark}
As we observed in Remark \ref{lambn}, one can construct other solutions $u_n$, $n=1,2,\dots$, by choosing the admissible values $\Lambda_n$, $n\ge 1$. The corresponding values of the total mass are
\[
\mu_n(z) = \mu(z) + n\underbrace{\sqrt3\int_{-u_{M}^*}^{u_{M}^*} \frac{v^2\,dv}{\sqrt{6H^*-v^6 + v^2}}}_{M^*},
\]
where $\mu(z)$ is defined in \eqref{lowmass}. Then, as $z$ spans $(0,+\infty)$, we get a sequence of allowed bands for the total masses.
Since $u^*_M\to1$ in both the limits $z\to0$ and $z\to+\infty$, it is readily checked that $M^*\to\mu_\R$ in those limits.

Finally, we mention that a second sequence of solutions can be found by gluing (at the origin in $\R$) the solitons in
\eqref{defsol} translated in the opposite direction, that is by defining
\[
\tilde u\big |_{\mathbb{R}}(x)=\left\{
                                                \begin{array}{ll}
                                                    \phi_{\Lambda}(x-y), & \hbox{if $x\ge 0$;} \\
                                                    \phi_{\Lambda}(x+y), & \hbox{if $x<0$,}
                                                  \end{array}
                                                \right.
\]
on $\R$ and $\tilde u(0) = \phi_\Lambda(y)$, $\tilde u'(0) = -2|\phi'_\Lambda(y)|$ on the pendant. This produces solutions with higher masses.
\end{remark}

\section{A compactness argument for locally minimizing sequences}\label{sec: cpt}

In this section we develop a general compactness argument for suitable locally minimizing sequences, which we shall directly apply in the proof of Theorem \ref{thm:main_intro}.

Let $\cG$ be a non-compact connected metric graph, with a non-empty compact core $K$, and having a finite number of vertices, bounded edges, and half-lines. Throughout this section, we denote by $\{\alpha_1, \alpha_2, \dots, \alpha_q\}$ the finite vertices of $\cG$, by $\{e_1,\dots,e_p\}$ its bounded edges, and by $\{\ell_1,\dots, \ell_m\}$ its half-lines. The vertices of $\cG$ are exactly $\{\alpha_1,\dots, \alpha_q\}$ plus the vertices at infinity. We identify $u \in H^1(\cG)$ with a vector $(u_1, \dots, u_m, v)$, where $u_i \in H^1(\ell_i) \simeq H^1(0,+\infty)$ is the restriction of $u$ on the half-line $\ell_i$, and $v \in H^1(K)$ is the restriction of $u$ on the compact core $K$. In turn, we denote by $v_i$ the restriction of $v$ on the edge $e_i$. If $e_i$ or $\ell_i$ is incident to $\alpha_j$, we write $e_i \succ \alpha_j$ or $\ell_i \succ \alpha_j$.

\medskip

Let us consider $0<\eta<\mu < \mu_{\cG}$, $\delta>0$ such that
\begin{equation}\label{hp dati}
\mu-\eta<\mu_{\R_+}, \quad \text{and} \quad (1+\delta)\eta<\mu.
\end{equation}
Let us also define
\begin{equation}\label{A B gen}
\begin{split}
\mathcal{A}& := \left\{ u \in H^1_\mu(\cG): \ \int_K u^2 \ge \eta \right\}, \\
\mathcal{B}& := \left\{ u \in H^1_\mu(\cG): \ \eta \le \int_K u^2 \le (1+\delta)\eta \right\}.
\end{split}
\end{equation}
Notice that $\mathcal{B}$ represents a neighborhood of $\pa \cA$ in $\cA$, in the $H^1$ topology. The main result of this section is the following variational principle.

\begin{proposition}\label{prop: comp}
Suppose that \eqref{hp dati} holds, and let
\begin{equation}\label{in min gen}
-\infty  < \inf_{u \in \cA} \ \En(u,\cG) < \inf_{u\in \cB} \ \En(u,\cG).
\end{equation}
Then $E(\cdot, \cG)$ constrained on $H^1_\mu(\cG)$ has a critical point which is obtained as local minimizer in the set $\cA$.
\end{proposition}

The proof is divided into several intermediate lemmas.

\begin{lemma}\label{lem:min-PS}
Let $\{\bar u_n\}$ be a minimizing sequence for $\En(\cdot\,,\cG)|_{\cA}$. Then there exists a (possibly different) minimizing sequence $\{u_n\}$ satisfying the additional properties that
\[
\|d\En(u_n,\cG)\|_* \to 0, \quad \text{and} \quad \|u_n- \bar{u}_n\|_{H^1_\mu(\cG)} \to 0
\]
as $n \to \infty$, where $\|\cdot\|_*$ denotes the norm in the dual to the tangent spaces $T_{u_n} (H^1_\mu(\cG))$.
\end{lemma}

We shall refer to $\{u_n\}$ as to a \emph{minimizing Palais-Smale} sequence for $\En(\cdot, \cG)|_{\cA}$.

\begin{proof}
This is a direct consequence of Ekeland's variational principle in the present setting, where we use assumption \eqref{in min gen} in order to ensure that minimizing sequences in $\cA$ do not approach the boundary $\pa \cA$.
\end{proof}

\begin{lemma}\label{lem: lim pos in 0}
Let $\{\bar{ u}_n\}$ be a minimizing sequence for $\En(\cdot\,,\cG)|_{\cA}$, with $\bar u_n \ge 0$ for every $n$. Then there exists a minimizing Palais-Smale sequence $\{u_n\}$, a function $u \in H^1(\cG)$, and $\lambda \in \R$ such that, up to a subsequence, we have
\begin{equation}\label{conv1}
\begin{cases}
\|u_n- \bar{u}_n\|_{H^1}  \to 0, \\ u_n, \bar{u}_n \weak u \quad \text{weakly in }H^1(\cG), \\ u_n, \bar{u}_n \to u \quad \text{locally uniformly in $\cG$},
\end{cases}
\end{equation}
as $n \to \infty$, and the limit $u=(u_1,\dots,u_m,v_1,\dots,v_p)$ satisfies
\begin{equation}\label{eq limit}
\begin{cases}
-u_i'' = \lambda u_i + u_i^5, \quad u_i > 0 & \text{on $\ell_i$, for every $i$} \\
-v_i'' = \lambda v_i + v_i^5, \quad v_i > 0 & \text{on $e_i$, for every $i$} \\
\sum_{e_i \succ \alpha_j} v_i'(\alpha_j) + \sum_{\ell_i \succ \alpha_j} u_i'(\alpha_j) = 0 & \text{for every $j =1, \dots, q$} \\
\end{cases}
\end{equation}
and at each vertex $\alpha_j$ we have that $u(\alpha_j) >0$.
\end{lemma}
Since each edge is identified with an interval $(0,d_i)$, $v_i'(\alpha_j)$ is a shorthand notation for $v_i'(0^+)$ or $-v_i'(d_i^-)$, according to the fact that the coordinate is $0$ or $d_i$ at $\alpha_j$. Similarly, writing that $u(\alpha_j) >0$ we mean that:
\[
\text{if } e_{i_1},\dots, e_{i_{p}}, \ell_{k_1}, \dots, \ell_{k_{q}} \succ \alpha_j, \quad \text{then} \quad v_{i_1}(\alpha_j) = \dots = v_{i_{p}}(\alpha_j) = u_{k_1}(\alpha_j) = \dots = u_{k_{q}}(\alpha_j) >0.
\]
\begin{proof}
By Lemma \ref{lem:min-PS}, there exists a minimizing Palais-Smale sequence $\{u_n\}$ such that $\|u_n - \bar{u}_n\|_{H^1} \to 0$. By \eqref{coerc1}, since $\mu<\mu_{\cG}$ we have that $\{u_n\}$ is bounded, and hence up to a subsequence we have
\[
\begin{split}
u_n \weak u \quad \text{weakly in }H^1(\cG), \\
u_n \to u \quad \text{locally uniformly in $\cG$}.
\end{split}
\]
Let now $\varphi \in H^1_\mu(\cG)$. Arguing as in \cite[Lemma 3]{BerLio2} (see also \cite{MR3456809}), the fact that $\{u_n\}$ is a bounded Palais-Smale sequence yields
\begin{equation}\label{weakeq}
d\En(u_n)[\varphi] - \lambda_n \int_{\cG} u_n \varphi  = o(1) \|\varphi\|_{H^1}
\end{equation}
as $n \to \infty$, for some approximate Lagrange multiplier $\lambda_n \in \R$, where $\lambda_n$ is given by
\[
\lambda_n = \frac{1}{\mu} \int_{\cG} |u_n'|^2 - |u_n|^6.
\]
It follows that $\{\lambda_n\}$ is bounded, and up to a further subsequence $\lambda_n \to \lambda \in \R$. Thus, by weak convergence, \eqref{weakeq} implies that \eqref{eq limit} holds (see \cite[Prop. 3.3]{AST1} for the details). Notice that, since  $\|u_n - \bar{u}_n\|_{H^1} \to 0$ and $\bar u_n \ge 0$, we have $u \ge 0$ as well. Now we show that $u(\alpha_j) > 0$ for every $j$. By contradiction, let $u(\alpha_j) = 0$ for some $j$, say $j=1$. Since by convergence $u \ge 0$, we have that
\[
v_i'(\alpha_1), u_k'(\alpha_1) \ge 0 \quad \text{whenever } e_i, \ell_k \succ \alpha_1.
\]
Thus, the Kirchhoff condition implies that in fact these derivatives are all equal to $0$, and by uniqueness of the solutions for the Cauchy problems associated with the NLS on intervals, we deduce that $v_i \equiv 0$ for $e_i \succ \alpha_1$, and $u_k \equiv 0$ for $\ell_k \succ \alpha_1$. This implies that $u(\alpha_j) = 0$ also for all the vertices directly connected to $\alpha_1$. Since the graph is connected and has finitely many vertices, iterating this argument a finite number of times, we infer that $u \equiv 0$ in $\cG$. However, this is not possible since $u_n \to u$ uniformly on compact sets, and in particular
\[
\int_K u^2 = \lim_{n \to \infty} \int_{K} u_n^2 \ge \eta>0.
\]
This contradiction shows that necessarily $u(\alpha_j) >0$ for every $j$, and this, by the strong maximum principle, finally gives the strict positivity of $u$.
\end{proof}

In the next lemma we select a particular minimizing sequence having some special symmetry properties, which will be useful in the study of its convergence.

\begin{lemma}\label{lem: spec PS}
There exists a minimizing sequence $\{\hat{u}_n = (\hat u_{1,n},\dots, \hat u_{m,n}, \hat v_n) \}$ for $\En(\cdot\,, \cG)$ on $\cA$, with $\hat{u}_n \ge 0$, and with the property that $\hat u_{i,n}$ is monotone decreasing on $\ell_i = (0,+\infty)$, for every $i=1,\dots,m$.
\end{lemma}

\begin{proof}
Let $\{\bar{u}_n =(\bar u_{1,n},\dots,\bar u_{m,n}, \bar v_n)\}$ be a minimizing sequence. If necessary replacing $\bar{u}_n$ with $(|\bar u_{1,n}|,\dots,|\bar u_{m,n}|, |\bar v_n|)$, we can suppose that $\bar{u}_n \ge 0$. By Lemma \ref{lem: lim pos in 0}, $u_n \to u$ locally uniformly on $\cG$, with $u$ positive in all the (finitely many) bounded vertices. Thus, for sufficiently large $n$, we have that $u_n$ is positive in all the vertices as well.

Let us consider $u_n^* = (u_{1,n}^*,\dots, u_{m,n}^*, \bar v_n)$, where $u_{i,n}^* \in H^1(\R^+)$ is the decreasing rearrangement of $\bar u_{i,n}$. By well known properties of rearrangements, we have that $u_n^* \in H^1(\R_+,\R^m) \times H^1(K)$, that $|u_n^*|_{L^2(\cG)}^2 = \mu$, and that $\En(u_n^*,\R) \le \En(\bar u_n,\R)$, for every $n$. However, in general $u_n^* \not \in H^1_\mu(\cG)$, since when we rearrange we could loose the continuity in the vertices. To overcome this problem, we observe that if $\alpha_j$ is the initial vertex of $\ell_i$, then
\[
u_{i,n}^*(0) = \sup_{\R^+} u_{i,n}^* = \sup_{\R^+} \bar u_{i,n} \ge \bar u_{i,n}(0) = \bar{{u}}_n(\alpha_j)>0.
\]
Therefore, we can consider
\[
\hat u_{i,n}(x):= \sqrt{\theta} u_{i,n}^*(\theta x),
\]
where $0 <\theta \le 1$ is chosen in such a way that $\hat u_{i,n}(0) = \bar{u}_n(\alpha_j)$. Moreover,
\[
\int_0^\infty (\hat u_{i,n})^2 = \int_0^\infty (u_{i,n}^*)^2 \le \int_{\cG} (u_n^*)^2- \int_K (u_n^*)^2 \le \mu -\eta <\mu_{\R^+},
\]
for every $i$, due to \eqref{hp dati}. In particular, we have that $\hat{u}_n
= (\hat u_{1,n},\dots, \hat u_{m,n}, \bar v_n) \in H^1_\mu(\cG)$, for every $n$; moreover, by \eqref{coerc1} applied to each restriction $u_{i,n}^*$, we have that $\En(u_{i,n}^*,\R^+) \ge 0$ and, since $\theta \le 1$,
\[
\En(\hat u_{i,n}, \R^+) = \theta^2 \En(u_{i,n}^*,\R^+) \le \En(u_{i,n}^*,\R^+).
\]
This means that $\{\hat{ u}_n\} \subset \cA$ is a minimizing sequence with the required properties.
\end{proof}

The advantage of working with $\{\hat{u}_n\}$ stays in the fact that its components $\hat u_{i,n}$ are decreasing functions on $\ell_i \simeq \R^+$, and the class $H_*(\R^+)$ of decreasing functions is compactly embedded in $L^p(\R^+)$, for every $2<p<+\infty$ (see the appendix in \cite{BerLio}). This yields the following lemma, which completes the proof of Proposition \ref{prop: comp}.

\begin{lemma}
Let $\{\hat{u}_n\}$ be the minimizing sequence given by Lemma \ref{lem: spec PS}. Then, up to a subsequence, $\hat{u}_n \to u$, and $u$ is a critical point of $\En(\cdot\,,\cG)$ on $H^1_\mu(\cG)$ obtained as local minimizer in $\cA$.
\end{lemma}

\begin{proof}
Let $\{\tilde{u}_n\}$ be the minimizing Palais-Smale sequence associated with $\{\hat{u}_n\}$, given by Lemma \ref{lem:min-PS}. By Lemma \ref{lem: lim pos in 0}, we can suppose that the two sequences converge to a limit $u \in H^1(\cG)$ (weakly in $H^1(\cG)$, and uniformly on compact sets of $\cG$), with $u$ satisfying \eqref{eq limit} and $u>0$ on $\cG$. \\
\textbf{Step 1)} Up to a subsequence, the limit is also strong in $L^6(\cG)$. Indeed, for each $i$ we recall that $\{\hat u_{i,n}\}$ is a bounded sequence in $H_*(\R^+)$, and hence, by compact embedding, $\hat u_{i,n} \to u_i$ strongly in $L^6(\R^+)$. Moreover, since $K$ is compact, we have also that $\hat v_n \to v$ strongly in $L^6(K)$, by local uniform convergence. \\
\textbf{Step 2)} The Lagrange multiplier $\lambda$ in \eqref{eq limit} is negative. Indeed, if by contradiction $\lambda \ge 0$, then $u_i$ would be a $C^2$ function on $(0,+\infty)$, concave, strictly positive (by the maximum principle), tending to $0$ as $x \to \infty$, which is not possible. \\ 
\textbf{Step 3)} $\tilde{u}_n \to u$ strongly in $H^1(\cG)$. Being $\tilde{u}_n$ a bounded Palais-Smale sequence, and recalling that $\lambda_n \to \lambda$, for any $\varphi \in H^1(\cG)$ we have that
\[
d\En(\tilde{u}_n, \cG)[\varphi] - \lambda \int_{\cG} \tilde{u}_n \varphi = o(1) \| \varphi\|_{H^1}
\]
as $n \to \infty$. Moreover
\[
d\En(u, \cG)[\varphi] - \lambda \int_{\cG} u \varphi = 0.
\]
Choosing $\varphi = \tilde{u}_n- u$ and subtracting, we deduce that
\[
\big(d\En(\tilde{u}_n, \cG) - d\En(u, \cG)\big)[\tilde{u}_n - u] - \lambda \int_{\cG} |\tilde{u}_n - u|^2 =o(1).
\]
But, having proved that $\tilde{u}_n \to u$ strongly in $L^6(\cG)$, the above equality reads
\[
\int_{\cG} |(\tilde{u}_n- u)'|^2 - \lambda |\tilde{u}_n- u|^2 = o(1),
\]
and, since $\lambda<0$, the left hand side is the square of a norm, equivalent to the standard one, in $H^1(\cG)$. Thus, we proved that $\tilde{u}_n \to u$ strongly in $H^1(\cG)$, and the thesis follows.
\end{proof}

\section{Proof of Theorem \ref{thm:main_intro}}\label{sec:proofmain}

Recalling the classification provided by \cite{AST2} and Proposition \ref{prop:removing}, we have that if either $\Gcal$ has exactly one half-line and no terminal point, or $\Gcal$ has no tips, no cycle-covering and at least $2$ half-lines, and $\mu_\Gcal<\mu_\R$, then $\cG$ can not fulfill the assumptions of Theorem \ref{thm:main_intro}. On the contrary, if $\Gcal$ is as in such theorem, only two possibilities are available, namely:
\begin{itemize}
\item $\mu_\Gcal = \mu_\Gcal\setminus\ell = \mu_{\R^+}$, for every open half-line
$\ell\subset\Gcal$. This is possible only if $\Gcal$ has at least two half-lines and a terminal point (and hence any $\Gcal\setminus\ell$ has at least a half-line and a terminal point, too). This case is treated in Section
\ref{subsec:with_tip}, see Proposition \ref{thm: tip gen} ahead.
\item  $\mu_\Gcal = \mu_\Gcal\setminus\ell = \mu_{\R}$, for every open half-line
$\ell\subset\Gcal$. This case is possible with different combinations, as explained in the introduction. In any case, by assumption, the compact core of $\Gcal$ is not empty. This case is treated in Section
\ref{subsec:no_tip}, see Proposition \ref{thm: notip gen} ahead.
\end{itemize}

\subsection{Non-compact graphs with a terminal point}\label{subsec:with_tip}

In this section we work under the following assumptions:
\begin{itemize}
\item[($g_1$)] $\mathcal{G}$ has a terminal point.
\item[($g_2$)] The non-compact part of $\cG$ consists in a finite number $m \ge 2$ of half-lines $\ell_1,\dots, \ell_m$.
\end{itemize}
The prototype of this class of graphs is the graph with the pendant considered in the previous sections. Notice that for any $\cG$ satisfying ($g_1$) and ($g_2)$ we have $\mu_{\cG}= \mu_{\R^+}$, and for the ground state energy level $\cE_{\cG}(\mu)$ we have:
\begin{itemize}
\item[(i)] If $\mu < \mu_{\R^+}$, then $\cE_{\cG}(\mu) = 0$, and is not attained.
\item[(ii)] If $\mu > \mu_{\R^+}$, then $\cE_{\cG}(\mu) =-\infty$,
\end{itemize}
see \cite[Corollary 2.5]{AST2}.

\begin{proposition}\label{thm: tip gen}
Let $\cG$ be a non-compact metric graph satisfying ($g_1$) and ($g_2$). Then there exists $\bar \mu \in (0, \mu_{\R^+})$ such that for every $\mu \in (\bar \mu, \mu_{\R^+})$ the functional $\En(\cdot, \cG)$ has a critical point on $H^1_{\mu}(\cG)$, which is a local minimizer.
\end{proposition}

The proof of the proposition will take the rest of the section. In view of Proposition \ref{prop: comp}, we shall conveniently introduce two sets $\cA$ and $\cB$ as in \eqref{A B gen}, and prove that for such sets \eqref{in min gen} holds. Precisely, for $\eta, \mu, \delta>0$ such that
\begin{equation}\label{hp dati tip}
(1+2\delta)\eta<\mu_{\R^+}, \quad \text{and} \quad  \mu \in [(1+2\delta)\eta,\mu_{\R^+}],
\end{equation}
we introduce
\[
\begin{split}
\cA_\eta^\mu & := \left\{u \in H^1_\mu(\cG)\left| \ \int_K u^2 \ge \eta \right.  \right\} \\
\cB_\eta^\mu & := \left\{u \in H^1_\mu(\cG)\left| \ \eta \le \int_K u^2 \le (1+\delta) \eta \right.\right\}.
\end{split}
\]
From now on, we shall always suppose that \eqref{hp dati tip} is in force, and observe that assumption \eqref{hp dati} in Proposition \ref{prop: comp} is trivially satisfied with this choice.

\begin{lemma}[Equicoercivity in $\cB_\eta^\mu$]\label{lem:eq-co}
There exists $C_1=C_1(\delta,\eta, m)>0$ (independent of $\mu$) such that
\[
\En(u,\cG) \ge C_1\|u'\|_{L^2(\cG)}^2
\]
for every $u \in \cB_\eta^\mu$, for every $\mu \in [(1+2\delta)\eta,\mu_{\R^+}]$.
\end{lemma}

\begin{proof}
If $u \in \cB_\eta^\mu$, then
\[
\sum_{i=1}^m \int_{\ell_i} u^2 = \mu - \int_{K} u^2 \in \left[\mu - (1+\delta) \eta, \mu - \eta\right].
\]
Therefore, there exists an index $\bar i \in \{1,\dots, m\}$, say $\bar i=m$, such that
\begin{equation}\label{3pag3}
\int_{\ell_m} u^2 \ge \frac1m\left( \mu - (1+\delta) \eta\right) \ge \frac{\delta \eta}{m}.
\end{equation}
Let $\mathcal{F}$ be the metric graph obtained removing the half-line $\ell_m$ and the corresponding vertex at infinity from $\cG$. By ($g_1$) and ($g_2$), this is a non-compact connected graph with at least one half-line and a terminal point, and hence $\mu_{\mathcal{F}} = \mu_{\R^+}$, by \cite[Theorem 3.1]{AST2}. In particular, \eqref{coerc1} gives that
\[
\En(u, \mathcal{F}) \ge \frac12 \left(1- \left(\frac{\int_{\mathcal{F}} u^2}{\mu_{\R^+}}\right)^2\right) \|u'\|_{L^2(\cF)}^2.
\]
Notice that, by \eqref{3pag3},
\[
\int_{\mathcal{F}} u^2 = \int_{\mathcal{G}} u^2 - \int_{\ell_m} u^2 \le  \mu-\frac{\delta\eta}{m} \le \mu_{\R^+} -\frac{\delta \eta}{m} ,
\]
which is strictly smaller than $\mu_{\R^+}$. Thus, it follows that for a constant $C_2(\delta,\eta,m)>0$
\begin{equation}\label{4pag3}
\En(u, \mathcal{F}) \ge C_2(\delta, \eta, m)  \|u'\|_{L^2(\cF)}^2.
\end{equation}
On the other hand, always by \eqref{coerc1}, on the half-line $\ell_m$
\[
\begin{split}
\En(u, \ell_m) &\ge  \frac12 \left(1- \left(\frac{\int_{\ell_m} u^2}{\mu_{\R^+}}\right)^2\right) \|u'\|_{L^2(\ell_m)}^2,
\end{split}
\]
and
\[
\int_{\ell_m} u^2 \le \mu - \int_{K} u^2 \le \mu-\eta \le \mu_{\R^+}-\eta,
\]
so that
\begin{equation}\label{5pag3}
\En(u, \ell_m) \ge C_3(\eta) \|u'\|_{L^2(\ell_m)}^2
\end{equation}
for some $C_3(\eta)>0$. Comparing \eqref{4pag3} and \eqref{5pag3}, we finally deduce that
\[
\En(u,\cG) = \En(u, \mathcal{F}) + \En(u, \ell_m) \ge \min\{C_2(\delta,\eta,m), C_3(\eta)\} \|u'\|_{L^2(\cG)}^2,
\]
which is the desired result with $C_1(\delta,\eta,m) = \min\{C_2(\delta, \eta, m), C_3(\eta)\}$.
\end{proof}

Using the equicoercivity with respect to $\mu$, it is not difficult to obtain the following uniform lower bounds.

\begin{lemma}[Uniform lower bound in $\cB_\eta^\mu$]\label{lem:l-b}
There exists $C_4=C_4(\delta,\eta,m)>0$ such that
 \[
\En(u,\cG) \ge C_4
\]
for every $u \in \cB_\eta^\mu$, for every $\mu \in [(1+2\delta)\eta,\mu_{\R^+}]$.
 \end{lemma}

 \begin{proof}
Suppose by contradiction that there exist sequences $\{\mu_n\} \subset [(1+2\delta)\eta,\mu_{\R^+}]$ and $u_n \in \cB_\eta^{\mu_n}$ such that $\En(u_n,\cG) \to 0$ as $n \to \infty$.
By Lemma \ref{lem:eq-co}, we infer that $\{u_n\}$ is bounded in $H^1(\cG)$, and moreover $\|u_n'\|_{L^2(\cG)} \to 0$ as $n \to \infty$; thus, up to a subsequence, we have that $u_n \weak u$ weakly in $H^1(\cG)$, and $u_n \to u$ locally uniformly on $\cG$, and by weak lower semi-continuity
 \begin{equation}\label{sn=0}
 \|u'\|_{L^2(\cG)}^2 \le \liminf_{n \to \infty} \|u_n'\|_{L^2(\cG)}^2 = 0;
 \end{equation}
this implies that $u$ is constant on $\cG$, and in fact, since $u \in H^1(\cG)$ and $\cG$ is non-compact, we have that necessarily $u \equiv 0$. However, by local uniform convergence
 \[
 \int_{K} u^2 = \lim_{n \to \infty}  \int_{K} u_n^2 \ge \eta >0,
 \]
a contradiction.
 \end{proof}

\begin{lemma}[Infimum in $\cA_\eta^\mu$]\label{lem:inf-a}
There exists $\bar \mu \in ((1+2\delta)\eta, \mu_{\R^+})$ such that
\[
\inf_{u \in \cA_\eta^\mu} \ \En(u,\cG) < C_4
\]
for every $\mu \in (\bar \mu, \mu_{\R^+})$.
\end{lemma}

\begin{proof}
Let $e$ be the edge of $\cG$ containing the terminal point. We identify $e$ with $[0,d]$, where the coordinate $0$ is taken in the terminal point. Also, in order to simplify some expressions and without loss of generality, we suppose that $d=1$. We show that, for any $\eps>0$ there exists $\mu_\eps \in ((1+2\delta)\eta, \mu_{\R^+})$ such that
\[
\mu \in (\mu_\eps, \mu_{\R^+}) \quad \implies \quad \exists w_\mu \in \cA_\eta^\mu \quad \text{with} \quad \En(w_\mu,\cG) <\eps.
\]
This in particular gives the thesis for $\eps = C_4$. For the exact choice of $w_\mu$, we consider the half-soliton $\phi$ (with $\phi$ defined in \eqref{defsol}), its scaling $\phi_\lambda(x):= \sqrt{\lambda} \phi(\lambda x)$, with $\lambda>0$, and we let
\[
w_\lambda(x):= (\phi_\lambda(x)- \phi_\lambda(1))^+.
\]
It is clear that $w_\lambda \in H^1(0,1)$, with $w_\lambda > 0$ on $[0,1)$, and $w_\lambda(1) = 0$. By monotone and dominated convergence, it is not difficult to check that the quantity
\[
m_\lambda:= \int_0^1 w_\lambda^2 =  \int_0^\infty (\phi(y) - \phi(\lambda))^2 \chi_{[0,\lambda]}(y)\,dy
\]
is continuous and monotone (strictly) increasing with respect to $\lambda$, and has limits
\[
\lim_{\lambda \to 0^+} m_\lambda = 0, \quad \text{and} \quad \lim_{\lambda \to +\infty} m_\lambda = \int_0^\infty \phi^2 = \mu_{\R^+}.
\]
Thus, for every $\mu \in (0, \mu_{\R^+})$ there exists a unique $\lambda(\mu)> 0$ such that $m_{\lambda(\mu)} = \mu$, and we define a function $w_\mu$ on the whole graph $\cG$ by setting
\[
w_{\mu} := \begin{cases} 0 & \text{on } \cG \setminus e \\ w_{\lambda(\mu)} & \text{on } e=[0,1].
\end{cases}
\]
It remains to check that $\En(w_\mu, \cG) = \En(w_{\lambda(\mu)},(0,1))$ can be made arbitrarily small as $\mu \to \mu_{\R^+}$, that is, as $\lambda \to +\infty$. We have
\[
\begin{split}
\En(w_{\lambda},(0,1)) & = \int_0^1 \frac12 (w_{\lambda}')^2 - \frac16 \int_0^1 w_\lambda^6 \\
& = \lambda^2 \left[ \int_0^\lambda \frac12 (\phi'(y))^2 -\frac16 (\phi(y)-\phi(\lambda))^6\right]dy \\
& \le \lambda^2 \left[ \int_0^\lambda \frac12 (\phi'(y))^2 -\frac16 \phi^6(y)\right]dy + \lambda^2 \phi(\lambda) \int_0^\lambda \phi^5(y)\,dy,
\end{split}
\]
where we used the fact that, for every $y  \in (0,\ \lambda)$, there exists $\tau_y \in (0,1)$ such that
\[
|(\phi(y)-\phi(\lambda))^6 - \phi(y)^6| = 6 |(\phi(y)-\tau_y \phi(\lambda))^5|\phi(\lambda) \le 6 \phi^5(y) \phi(\lambda).
\]
Clearly $\int_0^\lambda \phi^5 \le \int_0^\infty \phi^5 < +\infty$, and moreover $\lambda^2 \phi(\lambda) \to 0$ as $\lambda \to \infty$, by exponential decay; hence, the above estimate reads
\[
\En(w_{\lambda},(0,1)) = \lambda^2 \En(\phi, (0,\lambda)) + C o(1),
\]
as $\lambda \to +\infty$. But
\[
\En(\phi, (0,\lambda)) + \En(\phi, (\lambda,+\infty)) = \En(\phi,(0,+\infty)) = 0 \quad \implies \quad \En(\phi, (0,\lambda)) = - \En(\phi, (\lambda,+\infty)),
\]
whence
\[
\En(w_{\lambda},(0,1)) = -\lambda^2 \En(\phi, (\lambda,+\infty)) + C o(1)
\]
as $\lambda \to +\infty$. To proceed further, we observe that
\[
\En(\phi, (\lambda,+\infty)) = \En(\phi(\cdot + \lambda), (0,+\infty)) \ge 0,
\]
since clearly $\phi(\cdot+\lambda) \in H^1(\R^+)$ with $\int_0^\infty \phi^2(y + \lambda)\,dy \in (0,\mu_{\R^+})$ for every $\lambda >0$. Therefore
\[
\En(w_{\lambda},(0,1)) \le  C o(1) <\eps
\]
for every $\lambda>0$ sufficiently large. This completes the proof.
\end{proof}

We are finally ready for the:

\begin{proof}[Proof of Proposition \ref{thm: tip gen}]
Let $\bar \mu$ given by Lemma \ref{lem:inf-a}. For $\mu \in (\bar \mu, \mu_{\R^+})$, we let $\cA = \cA_\eta^\mu$ and $\cB = \cB_\eta^\mu$. Lemmas \ref{lem:l-b} and \ref{lem:inf-a} ensures that \eqref{in min gen} holds, so that Proposition \ref{prop: comp} directly gives the thesis.
\end{proof}

\subsection{Non-compact graphs without a terminal point}\label{subsec:no_tip}

In this section we consider non-compact connected metric graphs $\cG$ having a finite number of vertices, bounded edges, and half-lines $\ell_1,\dots,\ell_m$, and satisfying the following structural assumptions:
\begin{itemize}
\item[($h_1$)] $\mu_{\cG} = \mu_{\R}$, and $\mu_{\cG \setminus \ell_i} = \mu_{\R}$ for every $i=1,\dots,m$.
\item[($h_2$)] The compact core $K$ of $\cG$ is not empty.
\end{itemize}
Assumption ($h_1$) means that the critical mass of the graph obtained from $\cG$ removing an arbitrary half-line (and the corresponding vertex at infinity) is $\mu_{\R}$. In view of \cite[Theorems 3.1 and 3.3]{AST2}, this rules out the presence of terminal points, and implies also that $m \ge 3$.

For any $\cG$ satisfying ($h_1$) and ($h_2)$, we have:
\begin{itemize}
\item[(i)] If $\mu < \mu_{\R}$, then $\cE_{\cG}(\mu) = 0$, and is not attained.
\item[(ii)] If $\mu > \mu_{\R}$, then $\cE_{\cG}(\mu) =-\infty$,
\end{itemize}
see \cite[Corollary 2.5]{AST2}. Our main result for this class of graph is the following:

\begin{proposition}\label{thm: notip gen}
Let $\cG$ be a non-compact metric graph satisfying ($h_1$) and ($h_2$). Then there exists $\bar \mu \in (0, \mu_{\R})$ such that, for every $\mu \in (\bar \mu, \mu_{\R})$, the functional $\En(\cdot, \cG)$ has a critical point on $H^1_{\mu}(\cG)$, which is a local minimizer.
\end{proposition}

The proof of the proposition follows closely the one of Proposition \ref{thm: tip gen}. Precisely, for $\delta, \eta, \mu>0$ such that
\begin{equation}\label{hp dati notip}
\eta>\mu_{\R^+}, \quad (1+2\delta)\eta<\mu_{\R}, \quad \text{and} \quad  \mu \in [(1+2\delta)\eta,\mu_{\R}],
\end{equation}
we introduce
\[
\begin{split}
\cA_\eta^\mu & := \left\{u \in H^1_\mu(\cG)\left| \ \int_K u^2 \ge \eta \right.  \right\} \\
\cB_\eta^\mu & := \left\{u \in H^1_\mu(\cG)\left| \ \eta \le \int_K u^2 \le (1+\delta) \eta \right.\right\}.
\end{split}
\]
Again, \eqref{hp dati} is trivially satisfied by our choice in \eqref{hp dati notip}.

\begin{lemma}[Equicoercivity in $\cB_\eta^\mu$]\label{lem:eq-co'}
There exists $C_1=C_1(\delta, \eta, m)>0$ (independent of $\mu$) such that
\[
\En(u,\cG) \ge C_1\|u'\|_{L^2(\cG)}^2
\]
for every $u \in \cB_\eta^\mu$, for every $\mu \in [(1+2\delta)\eta,\mu_{\R}]$.
\end{lemma}

\begin{proof}
If $u \in \cB_\eta^\mu$, then
\[
\sum_{i=1}^m \int_{\ell_i} u^2 = \mu - \int_{K} u^2 \in \left[\mu - (1+\delta) \eta, \mu - \eta\right].
\]
Therefore, there exists an index $\bar i \in \{1,\dots, m\}$, say $\bar i=m$, such that
\begin{equation}\label{3pag3'}
\int_{\ell_m} u^2 \ge \frac1m\left( \mu - (1+\delta) \eta \right) \ge \frac{\delta \eta}{m}.
\end{equation}
Let $\mathcal{F}$ be the metric graph obtained removing the half-line $\ell_m$ and the corresponding vertex at infinity from $\cG$. By ($h_1$), we know that $\mu_{\mathcal{F}} = \mu_{\R}$. In particular, \eqref{coerc1} gives that
\[
\En(u, \mathcal{F}) \ge \frac12 \left(1- \left(\frac{\int_{\mathcal{F}} u^2}{\mu_{\R}}\right)^2\right) \|u'\|_{L^2(\cF)}^2.
\]
Notice that, by \eqref{3pag3'},
\[
\int_{\mathcal{F}} u^2 = \int_{\mathcal{G}} u^2 - \int_{\ell_m} u^2 \le \mu - \frac{\delta \eta}{m} \le \mu_{\R} - \frac{\delta \eta}{m},
\]
which is strictly smaller than $\mu_{\R}$. Thus, it follows that for a constant $C_2(\delta,\eta,m)>0$
\begin{equation}\label{4pag3'}
\En(u, \mathcal{F}) \ge C_2(\delta,\eta,m)  \|u'\|_{L^2(\cF)}^2.
\end{equation}
On the other hand, always by \eqref{coerc1}, on the half-line $\ell_m$
\[
\begin{split}
\En(u, \ell_m) &\ge  \frac12 \left(1- \left(\frac{\int_{\ell_m} u^2}{\mu_{\R^+}}\right)^2\right) \|u'\|_{L^2(\ell_m)}^2 \end{split}
\]
and
\[
\int_{\ell_m} u^2 \le \mu - \int_{K} u^2 \le \mu-\eta \le \mu_{\R}-\eta < \mu_{\R^+},
\]
since $\eta  >\mu_{\R^+}$, so that
\begin{equation}\label{5pag3'}
\En(u, \ell_m) \ge C_3(\eta) \|u'\|_{L^2(\ell_m)}^2
\end{equation}
for some $C_3(\eta)>0$. Comparing \eqref{4pag3'} and \eqref{5pag3'}, we finally deduce that
\[
\En(u,\cG) = \En(u, \mathcal{F}) + \En(u, \ell_m) \ge \min\{C_2(\delta,\eta,m), C_3(\eta)\} \|u'\|_{L^2(\cG)}^2,
\]
which is the desired result.
\end{proof}

 \begin{lemma}[Uniform lower bound in $\cB_\eta^\mu$]\label{lem:l-b'}
There exists $C_4=C_4(\delta,\eta,m)>0$ such that
 \[
\En(u,\cG) \ge C_4
\]
for every $u \in \cB_\eta^\mu$, for every $\mu \in [(1+2\delta)\eta,\mu_{\R}]$.
 \end{lemma}

\begin{proof}
The proof is completely analogue to the one of Lemma \ref{lem:l-b}, and hence is omitted.
\end{proof}

\begin{lemma}[Infimum in $\cA_\eta^\mu$]\label{lem:inf-a'}
There exists $\bar \mu \in ((1+2\delta)\eta, \mu_{\R})$ such that
\[
\inf_{u \in \cA_\eta^\mu} \ \En(u,\cG) < C_4
\]
for every $\mu \in (\bar \mu, \mu_{\R})$.
\end{lemma}

\begin{proof}
Since $K \neq \emptyset$ by assumption ($h_2$), and $\cG$ does not have terminal points, there exists a bounded edge $e$, of length $d>0$, such that both the vertices of $e$ are not terminal points. Without loss of generality, we suppose that $d=2$, and we identify $e$ with the interval $(-1,1)$. As in Lemma \ref{lem:inf-a}, we show that for any $\eps>0$ there exists $\mu_\eps \in ((1+2\delta)\eta, \mu_{\R})$ such that
\[
\mu \in (\mu_\eps, \mu_{\R}) \quad \implies \quad \exists w_\mu \in \cA_\eta^\mu \quad \text{with} \quad \En(w_\mu,\cG) <\eps.
\]
For the exact choice of $w_\mu$, we consider the soliton $\phi$, its scaling $\phi_\lambda(x):= \sqrt{\lambda} \phi(\lambda x)$, with $\lambda>0$, and we let
\[
w_\lambda(x):= (\phi_\lambda(x)- \phi_\lambda(1))^+.
\]
It is clear that $w_\lambda \in H^1_0(-1,1)$, with $w_\lambda > 0$ on $(-1,1)$, and $w_\lambda$ is symmetric with respect to $0$ (the medium point of the edge $e$) for every $\lambda>0$. Since
\[
\|w_\lambda\|_{L^p(-1,1)}^2 = 2\|w_\lambda\|_{L^p(0,1)}^2, \quad \text{and} \quad \|w_\lambda'\|_{L^2(-1,1)}^2 = 2\|w_\lambda'\|_{L^2(0,1)}^2,
\]
the same computations of Lemma \ref{lem:inf-a} allow to show that for every $\lambda \in (0,\mu_{\R})$ there exists a unique $\lambda(\mu) >0$ such that the function
\[
w_\mu:= \begin{cases}  w_{\lambda(\mu)} & \text{on $e$} \\ 0 & \text{in $\cG \setminus e$} \end{cases}
\]
stays in $H^1_\mu(\cG)$, and moreover $E(w_\mu, \cG) = E(w_\mu,(-1,1)) \to 0$ as $\mu \to (\mu_{\R})^-$.
\end{proof}

\begin{proof}[Proof of Proposition \ref{thm: notip gen}]
Let $\bar \mu$ given by Lemma \ref{lem:inf-a'}. For $\mu \in (\bar \mu, \mu_{\R})$, we let $\cA = \cA_\eta^\mu$ and $\cB = \cB_\eta^\mu$. By Lemmas \ref{lem:l-b'} and \ref{lem:inf-a'}, Proposition \ref{prop: comp} directly gives the thesis of the proposition.
\end{proof}

\begin{remark}\label{rem:stargraph}
As observed in the introduction, if $\cG$ is a star-graph with at least $3$ half-lines, then $\cG$ has not a bounded edge, and hence Theorem \ref{thm:main_intro} does not apply. It would be tempting in this case to fix an arbitrary compact set $K \subset \cG$ made of subintervals of each half-line, define $\cA_\eta^\mu$ and $\cB_\eta^\mu$ as above, and then mimic the proof of Proposition \ref{thm: notip gen}. Notice that the choice of $K$ induces the introduction of new fake vertices in the graph, with degree $2$. This strategy however cannot work, and in particular our equicoercivity lemma fails. Indeed, by centering and shrinking a soliton near one of the fake vertices, it is not difficult to see that the infimum of the energy on the corresponding set $\cB_\eta^\mu$ tends to $0$, in sharp contrast with the case when $\cG$ has a true compact core. This new phenomenon is possible exactly since the new vertices have degree $2$.
\end{remark}

\appendix

\section{Stability properties of the model solutions}\label{app:A}

In this appendix we sketch the proof that the solutions $u=u_z$ constructed in Section \ref{soltip} for the model graph, having total mass $\mu=\mu(z)$ as in \eqref{lowmass}, are (conditionally) orbitally stable  when the parameter $z$ is large enough. As we already mentioned, this follows
by the abstract theory of Grillakis, Shatah and Strauss, see
\cite[Thms. 2-3]{GrillakisShatahStrauss}.
Roughly, we will consider the general abstract structure, focusing on the Morse index of the solutions, in \ref{app:A1}, while in \ref{app:A2} we will deal with the monotonicity properties of the map $z\mapsto\mu(z)$.

\subsection{The abstract setting}\label{app:A1}

For the sake of comparison with the notations used in \cite{GrillakisShatahStrauss}, we denote by
$X=H^1(\mathcal{G})$ the space of the {complex valued} functions which are $H^1$ on every edge of the graph and continuous at every vertex, endowed with the usual (real) inner product
(see \cite[Section 6C]{GrillakisShatahStrauss}). Denoting the mass (charge) $Q$ as
\begin{equation}
Q(u)=-\frac{1}{2}\int_{\mathcal{G}}|u|^2
\end{equation}
we have that the {bound state equation} $E'(\phi_{\omega})-\omega Q'(\phi_{\omega})=0$ (with $\omega\in\R$, see \cite[Assumption 2]{GrillakisShatahStrauss}) reads
\begin{equation}
\label{bsteq}
-\phi''-|\phi|^4\phi+\omega\phi=0,\qquad \text{where }\phi=\phi_{\omega}\in H^1(\mathcal{G})
\end{equation}
is real and satisfies the Kirchhoff boundary conditions at every vertex. Restricting to the family of solutions constructed in Section \ref{soltip}, parametrized on $z\in\R^+$, we have that
\[
\omega=\frac{\Lambda^2(z)}3,\qquad \phi_\omega = u_z.
\]
At this point, assumptions $1$ and $2$ of the abstract theory of \cite{GrillakisShatahStrauss}
hold true in a standard way. Then, we can define the functional
\begin{equation}
\label{functd}
d(\omega)=E(\phi)-\omega Q(\phi)\,,
\end{equation}
and the operator from $X$ to $X^*$
\begin{equation}
\label{linop}
H_{\omega}=E''(\phi)-\omega Q''(\phi)\,.
\end{equation}
We consider the (decoupled) {eigenvalue equation} $H_{\omega}\,\chi=\nu\,\chi$, $\chi = \chi_1+ i
\chi_2$,
which writes, on every edge of the graph,
\begin{align}
\label{realsys}
  -\chi_1''-5\phi^4\,\chi_1+\frac{\Lambda^2}{3}\,\chi_1&= \nu\,\chi_1, \\
\label{imsys}
  -\chi_2''-\phi^4\,\chi_2+\frac{\Lambda^2}{3}\,\chi_2 &= \nu\,\chi_2,
\end{align}
where both $\chi_1$ and $\chi_2$ satisfy the usual continuity and Kirchhoff conditions at each vertex. It is readily verified that, if $\phi=u_z$, then
\[
\nu=0,\ \chi=i \phi
\]
satisfy the above eigenvalue problem. Notice that, since $\phi>0$, there are no other nontrivial solutions of \eqref{imsys} for $\nu\le0$, but $\chi_2=\phi$. As a consequence, any other eigenfunction with non-positive eigenvalue has to be real valued. On the other hand, the negative part of the spectrum of $H_\omega$ is not empty, since
\[
\langle H_{\omega}\phi,\phi\rangle= -4 \int_{\mathcal{G}}|\phi|^6<0.
\]
Therefore, in order to satisfy assumption 3 of \cite{GrillakisShatahStrauss}, we need to show that the subspace of eigenfunctions of \eqref{realsys} with non-positive eigenvalues has dimension 1. On the contrary,  there exists a real valued $\chi_\omega$ satisfying
\begin{equation}
\label{firstcond}
\langle H_{\omega}\phi,\chi_{\omega}\rangle =-4\int_{\mathcal{G}}\phi^5\chi_{\omega}=0
\end{equation}
and
\begin{equation}
\label{seccond}
\langle H_{\omega}\chi_{\omega},\chi_{\omega}\rangle=\int_{\mathcal{G}}|\chi'_{\omega}|^2-5\phi^4|\chi_{\omega}|^2 +\frac{\Lambda^2}{3}|\chi_{\omega}|^2\le0\,.
\end{equation}
We may take $\chi_{\omega}$ of unit total mass. We will prove that the above conditions are incompatible in the limit $\omega\to\infty$ (i.e. $\Lambda\to\infty$).
Recall now that on the positive (negative) real line
$$\phi(x)=\sqrt{\Lambda}\,\phi_1(\Lambda x+z)\,\,\,\big (\sqrt{\Lambda}\,\phi_1(\Lambda x-z)\big )\,,$$
where $\phi_1(x)=\cosh^{-1/2}(2x/\sqrt 3)$ and $\Lambda=\Lambda(z)$ is given by \eqref{lamz}. Since $\Lambda(z)\approx c z$ at infinity, we have that
$\phi\big |_{\R}\rightarrow 0$
uniformly for $z\to +\infty$. It follows by \eqref{firstcond} that
\begin{equation}
\label{firstcondlim}
\lim_{z\to +\infty}\int_0^1\phi^5(x)\,\chi_{\omega(z)}(x)\,dx=0\,.
\end{equation}
To deduce information from this limit, we introduce the new variable $\xi=\phi/\sqrt\Lambda$ which is related to $x$  by
\begin{equation}
x(\xi)=\frac{\sqrt 3}{\Lambda}\int_{\phi^*}^{\xi}\frac{dv}{\sqrt{6E^*-v^6+v^2}},\quad\quad \phi^*\le\xi\le u_M^*\,,
\end{equation}
where $\Lambda=\Lambda(z)$ is given again by \eqref{lamz} and $E^*$, $\phi^*$, $u_M^*$, are the functions of $z$ defined in \eqref{stars} and satisfying $E^*\to 0$, $\phi^*\to 0$, $u_M^*\to 1$ for $z\to +\infty$. Then, for every $\xi\in (\phi^*,u_M^*]$,
\begin{equation}
x(\xi)=1-\frac{\sqrt 3}{\Lambda}\int_{\xi}^{u_M^*}\frac{dv}{\sqrt{6E^*-v^6+v^2}}\longrightarrow 1
\end{equation}
for $z\to \infty$. In the same limit, we have
\begin{equation}
\int_0^1\phi^5(x)\,\chi_{\omega(z)}(x)\,dx=
{\sqrt 3}{\Lambda}^{3/2}(z)\int_{\phi^*}^{u_M^*}\frac{\xi^5\,\chi_{\omega(z)}(x(\xi))}{\sqrt{6E^*-\xi^6+\xi^2}}\,d\xi\,
\end{equation}
\begin{equation}
={\sqrt 3}{\Lambda}^{3/2}(z)\Big (\chi_{\omega(z)}(1)\int_{0}^{1}\frac{\xi^4}{\sqrt{1-\xi^4}}\,d\xi\,+o(1)\Big )\,.
\end{equation}
By \eqref{firstcondlim} we get
\begin{equation}
\lim_{z\to +\infty}\chi_{\omega(z)}(1)=0\,.
\end{equation}
By evaluating in the same way the integral in \eqref{seccond}, we find
\begin{equation}
5\int_{\mathcal{G}}|\phi(x)|^4\,|\chi_{\omega}(x)|^2=o(1)+
5\int_0^1|\phi(x)|^4\,|\chi_{\omega}(x)|^2={\Lambda}(z)\Big (C\,\chi_{\omega(z)}(1)^2+o(1)\Big )\le \Lambda(z)\,
\end{equation}
for $z$ large enough. But the last term in \eqref{seccond} is equal to $\Lambda^2/3$ (by the
normalization of $\chi_{\omega}$) so that such condition can not hold for large $z$. By this
contradiction, we have that assumption 3 in \cite{GrillakisShatahStrauss} holds true.

\subsection{Monotonicity properties of \texorpdfstring{$\mu(z)$}{mu(z)}}\label{app:A2}

In order to apply \cite[Thms. 2-3]{GrillakisShatahStrauss}, the last condition we need to check is that the map
$\omega \mapsto d(\omega)$, as defined in \eqref{functd}, is convex for large $\omega$.
We recall that
$$
d'(\omega)=-Q(\phi_\omega)=\frac{1}{2}\int_{\mathcal{G}}|\phi_\omega|^2
=\frac{1}{2}\int_{\mathcal{G}}|u_z|^2
$$
(see \cite[Eq.(2.20)]{GrillakisShatahStrauss}). On the other hand, $\omega=\Lambda^2(z)/3$, and
it can be shown by \eqref{lamz} that the map $z\mapsto \Lambda(z)$  is increasing to $+\infty$ as
$z\to+\infty$. As a consequence, we are left to show that the function
$$
\mu(z):=\int_{|x|>z} \frac{1}{\cosh(2 x/\sqrt 3)}\,dx+\sqrt{3} \int_{\phi^*}^{u_M^*}\frac{v^2}{\sqrt{6H^*-v^6+v^2}}\, dv\,
$$
defined in \eqref{lowmass}, is increasing for large values of $z$ (as it was shown, numerically, in Fig. \ref{fig:num}).

To this aim, we first recall the relations
\begin{equation}
\label{setstars}
\phi^*(z)=\cosh(2z/\sqrt 3)^{-1/2},\quad\quad
H^*(z)=\frac{1}{2}\phi^*(z)^2\big (1- \phi^*(z)^4\big )=\frac{1}{6}u_M^*(z)^2\big (u_M^*(z)^4-1\big )\,,
\end{equation}
and
\begin{equation}
\label{derstars}
\begin{split}
(\phi^*)'(z)&=-\frac{1}{\sqrt 3}\phi^*(z)\big (1- \phi^*(z)^4\big )^{1/2},\\
(H^*)'(z)&=-\frac{1}{\sqrt 3}\phi^*(z)^2\big (1- \phi^*(z)^4\big )^{1/2}\big (1- 3\phi^*(z)^4\big )\,.
\end{split}
\end{equation}
Moreover, from the last identity in \eqref{setstars} we deduce that
\begin{equation}
\label{asistars}
u_M^*(z)^4-1 \approx 3\phi^*(z)^2
\end{equation}
as $\phi^*\to 0$ (since $u_M^*\to 1^+$ necessarily). We want to estimate the derivative
\begin{equation}
\label{dermu0}
\mu'(z):=-\frac{2}{\cosh(2z/\sqrt 3)}+\sqrt{3}\frac{d}{dz} \int_{\phi^*}^{u_M^*}\frac{v^2}{\sqrt{6H^*-v^6+v^2}}\, dv\,
\end{equation}
for $z\to +\infty$. To this aim, by recalling that both $\phi^*\to 0^+$ and $H^*\to 0^+$ in this limit, we fix two positive numbers $v_0$, $v_1$ such that
$$
\phi^*<v_0<3^{-1/4}<v_1<1\,,
$$
and split the integral at the right hand side of \eqref{dermu0} into three pieces
$$
\int_{\phi^*}^{v_0}\frac{v^2}{\sqrt{6H^*-v^6+v^2}}\, dv\,+\int_{v_0}^{v_1}\frac{v^2}{\sqrt{6H^*-v^6+v^2}}\, dv\,+\int_{v_1}^{u_M^*}\frac{v^2}{\sqrt{6H^*-v^6+v^2}}\, dv.
$$
The limit of the derivative of the second integral is readily evaluated:
$$
\frac{d}{dz}\int_{v_0}^{v_1}\frac{v^2}{\sqrt{6H^*-v^6+v^2}}\, dv\,=-3\, (H^*)'(z)\,\int_{v_0}^{v_1}\frac{v^2}{(6H^*-v^6+v^2)^{3/2}}\, dv\,\approx
$$
(by \eqref{derstars})
\begin{equation}
\label{dersec}
\approx
\sqrt 3\,\phi^*(z)^2\,\int_{v_0}^{v_1}\frac{v}{(1-v^4)^{3/2}}\, dv\,.
\end{equation}
Before taking the derivative of the remaining terms, we change the variable of integration by setting
$$t=v^6-v^2,\quad\quad\quad dt=2v(3v^4-1)\,dv\,.$$
Then,  the first integral will have the limits $t( \phi^*)=-2H^*$ and $t(v_0):=-t_0$ ($<-2H^*$) :
$$
\int_{\phi^*}^{v_0}\frac{v^2}{\sqrt{6H^*-v^6+v^2}}\, dv\,=
\int_{-t_0}^{-2H^*}\frac{1}{2\sqrt{6H^*-t}}\,\frac{v(t)}{1-3v(t)^4}\, dt=
$$
(integrating by parts)
$$
=\Big [-\sqrt{6H^*-t}\,\frac{v(t)}{1-3v(t)^4}\Big ]_{-t_0}^{-2H^*}+
\int_{-t_0}^{-2H^*}\sqrt{6H^*-t}\,\frac{9v(t)^4+1}{(1-3v(t)^4)^2} v'(t)\, dt
$$
$$
=-\sqrt{8H^*}\,\frac{\phi^*}{1-3(\phi^*)^4}+\sqrt{6H^*+t_0}\,\frac{v_0}{1-3v_0^4}
-\int_{\phi^*}^{v_0}\sqrt{6H^*-v^6+v^2}\,\frac{9v^4+1}{(1-3v^4)^2}\, dv\,.
$$
Now, by elementary calculations one checks that in the limit $z\to +\infty$ the derivatives of the above terms are
$\mathcal{O}((\phi^*)^2)$ (or even ${o}((\phi^*)^2)$) \emph{except for}
\begin{multline*}
-3(H^*)'(z)\,\int_{\phi^*}^{v_0}\frac{1}{\sqrt{6H^*-v^6+v^2}}\,\frac{9v^4+1}{(1-3v^4)^2}\, dv\,
\\
\approx
\sqrt 3\,\phi^*(z)^2\,\int_{\phi^*}^{v_0}\frac{1}{\sqrt{6H^*-v^6+v^2}}\,\frac{9v^4+1}{(1-3v^4)^2}\, dv\,.
\end{multline*}
In order to estimate the last integral we write as before
$$
6H^*-v^6+v^2=\big [(u_M^*)^2-v^2\big ]\,\big [(u_M^*)^4-1+(u_M^*)^2v^2+v^4\big ]\approx
$$
(by \eqref{asistars})
$$
\approx \big [(u_M^*)^2-v^2\big ]\,\big [3(\phi^*)^2+(u_M^*)^2v^2+v^4\big ]\le
\big [(u_M^*)^2-v^2\big ]\,\big [3v^2+(u_M^*)^2v^2+v^4\big ]\le C v^2
$$
for every $v\in [\phi^*,v_0]$, with $C>0$. Hence, we get
\begin{equation}
\label{logest}
\int_{\phi^*}^{v_0}\frac{1}{\sqrt{6H^*-v^6+v^2}}\,\frac{9v^4+1}{(1-3v^4)^2}\, dv\,\ge \frac{1}{\sqrt C}
\int_{\phi^*}^{v_0}\frac{dv}{v}= \frac{1}{\sqrt C} \ln\Big (\frac{v_0}{\phi^*}\Big )\,.
\end{equation}
We conclude that, for $z\to +\infty$,
\begin{equation}
\label{derprim}
\frac{d}{dz}\int_{\phi^*}^{v_0}\frac{v^2}{\sqrt{6H^*-v^6+v^2}}\, dv\,\ge \sqrt{\frac{3}{C}}
\phi^*(z)^2\ln\Big (\frac{1}{\phi^*(z)}\Big )+\mathcal{O}\big (\phi^*(z)^2\big )\,.
\end{equation}
Finally, by the same change of variable in the third integral, we get $t(v_1):=-t_1$,
$t(u_M^*)=6H^*$ and
$$
\int_{v_1}^{u_M^*}\frac{v^2}{\sqrt{6H^*-v^6+v^2}}\, dv\,=\int_{-t_1}^{6H^*}\frac{1}{2\sqrt{6H^*-t}}\,\frac{v(t)}{3v(t)^4-1}\, dt
$$
$$
=\Big [-\sqrt{6H^*-t}\,\frac{v(t)}{3v(t)^4-1}\Big ]_{-t_1}^{6H^*}-
\int_{-t_1}^{6H^*}\sqrt{6H^*-t}\,\frac{9v(t)^4+1}{(3v(t)^4-1)^2} v'(t)\, dt
$$
$$
=\sqrt{6H^*+t_1}\,\frac{v_1}{3v_1^4-1}
-\int_{v_1}^{u_M^*}\sqrt{6H^*-v^6+v^2}\,\frac{9v^4+1}{(3v^4-1)^2}\, dv\,.
$$
Hence,
\begin{multline*}
\frac{d}{dz}\int_{v_1}^{u_M^*}\frac{v^2}{\sqrt{6H^*-v^6+v^2}}\, dv\,
\\
=\frac{3(H^*)'(z)}{\sqrt{6H^*+t_1}}\,
\frac{v_1}{3v_1^4-1}-3(H^*)'(z)\,\int_{v_1}^{u_M^*}\frac{1}{\sqrt{6H^*-v^6+v^2}}\,\frac{9v^4+1}{(3v^4-1)^2}\, dv\,
\end{multline*}
Since both the coefficients multiplying the term $(H^*)'(z)$ are finite for $\phi^*(z)\to 0$  ($H^*\to 0$) the above derivative is
$\mathcal{O}\big (\phi^*(z)^2\big )$ for $z\to +\infty$.
By recalling the definition \eqref{dermu0}, we conclude that
$$
\lim_{z\to +\infty}\mu'(z)= 0^+.
$$

\section*{Acknowledgments}
Work partially supported
by the PRIN-2015KB9WPT Grant: ``Variational methods, with applications to problems in
mathematical physics and geometry'', and by the INdAM-GNAMPA group.



\begin{thebibliography}{10}

\bibitem{AST1}
R.~Adami, E.~Serra, and P.~Tilli.
\newblock N{LS} ground states on graphs.
\newblock {\em Calc. Var. Partial Differential Equations}, 54(1):743--761,
  2015.

\bibitem{MR3494248}
R.~Adami, E.~Serra, and P.~Tilli.
\newblock Threshold phenomena and existence results for {NLS} ground states on
  metric graphs.
\newblock {\em J. Funct. Anal.}, 271(1):201--223, 2016.

\bibitem{AST2}
R.~Adami, E.~Serra, and P.~Tilli.
\newblock Negative energy ground states for the {$L^2$}-critical {NLSE} on
  metric graphs.
\newblock {\em Comm. Math. Phys.}, 352(1):387--406, 2017.

\bibitem{AST11}
R.~Adami, E.~Serra, and P.~Tilli.
\newblock Nonlinear dynamics on branched structures and networks.
\newblock {\em Riv. Math. Univ. Parma (N.S.)}, 8(1):109--159, 2017.

\bibitem{AST10}
R.~Adami, E.~Serra, and P.~Tilli.
\newblock Multiple positive bound states for the subcritical {NLS} equation on
  metric graphs.
\newblock {\em Calc. Var. Partial Differential Equations}, 58(1):Paper No. 5,
  16, 2019.

\bibitem{MR3638314}
J.~Bellazzini, N.~Boussa\"id, L.~Jeanjean, and N.~Visciglia.
\newblock Existence and stability of standing waves for supercritical {NLS}
  with a partial confinement.
\newblock {\em Comm. Math. Phys.}, 353(1):229--251, 2017.

\bibitem{BellazziniGeorgievVisciglia}
J.~Bellazzini, V.~Georgiev, and N.~Visciglia.
\newblock Long time dynamics for semi-relativistic {NLS} and half wave in
  arbitrary dimension.
\newblock {\em Math. Ann.}, 371(1-2):707--740, 2018.

\bibitem{BellazziniJeanjean}
J.~Bellazzini and L.~Jeanjean.
\newblock On dipolar quantum gases in the unstable regime.
\newblock {\em SIAM J. Math. Anal.}, 48(3):2028--2058, 2016.

\bibitem{BerLio}
H.~Berestycki and P.-L. Lions.
\newblock Nonlinear scalar field equations. {I}. {E}xistence of a ground state.
\newblock {\em Arch. Rational Mech. Anal.}, 82(4):313--345, 1983.

\bibitem{BerLio2}
H.~Berestycki and P.-L. Lions.
\newblock Nonlinear scalar field equations. {II}. {E}xistence of infinitely
  many solutions.
\newblock {\em Arch. Rational Mech. Anal.}, 82(4):347--375, 1983.

\bibitem{BK}
G.~Berkolaiko and P.~Kuchment.
\newblock {\em Introduction to quantum graphs}, volume 186 of {\em Mathematical
  Surveys and Monographs}.
\newblock American Mathematical Society, Providence, RI, 2013.

\bibitem{Cazenave2003}
T.~Cazenave.
\newblock {\em Semilinear {S}chr\"odinger equations}, volume~10 of {\em Courant
  Lecture Notes in Mathematics}.
\newblock New York University Courant Institute of Mathematical Sciences, New
  York, 2003.

\bibitem{MR3758538}
S.~Dovetta.
\newblock Existence of infinitely many stationary solutions of the
  {$L^2$}-subcritical and critical {NLSE} on compact metric graphs.
\newblock {\em J. Differential Equations}, 264(7):4806--4821, 2018.

\bibitem{Dnodea}
S.~Dovetta.
\newblock Mass-constrained ground states of the stationary {NLSE} on periodic
  metric graphs.
\newblock {\em NoDEA Nonlinear Differential Equations Appl.}, 26(5):Paper No.
  30, 30, 2019.

\bibitem{DTcalcvar}
S.~Dovetta and L.~Tentarelli.
\newblock {$L^2$}-critical {NLS} on noncompact metric graphs with localized
  nonlinearity: topological and metric features.
\newblock {\em Calc. Var. Partial Differential Equations}, 58(3):Paper No. 108,
  26, 2019.

\bibitem{GrillakisShatahStrauss}
M.~Grillakis, J.~Shatah, and W.~Strauss.
\newblock Stability theory of solitary waves in the presence of symmetry, {I}.
\newblock {\em Journal of Functional Analysis}, 74(1):160--197, 1987.

\bibitem{MR1081647}
M.~Grillakis, J.~Shatah, and W.~Strauss.
\newblock Stability theory of solitary waves in the presence of symmetry. {II}.
\newblock {\em J. Funct. Anal.}, 94(2):308--348, 1990.

\bibitem{No}
D.~Noja.
\newblock Nonlinear {S}chr\"{o}dinger equation on graphs: recent results and
  open problems.
\newblock {\em Philos. Trans. R. Soc. Lond. Ser. A Math. Phys. Eng. Sci.},
  372(2007):20130002, 20, 2014.

\bibitem{NP}
D.~Noja and D.~E. Pelinovsky.
\newblock Standing waves of the quintic {NLS} equation on the tadpole graph.
\newblock {\em Preprint arXiv:2001.00881}, 2020.

\bibitem{ntvAnPDE}
B.~Noris, H.~Tavares, and G.~Verzini.
\newblock Existence and orbital stability of the ground states with prescribed
  mass for the {$L^2$}-critical and supercritical {NLS} on bounded domains.
\newblock {\em Anal. PDE}, 7(8):1807--1838, 2014.

\bibitem{ntvDCDS}
B.~Noris, H.~Tavares, and G.~Verzini.
\newblock Stable solitary waves with prescribed {$L^2$}-mass for the cubic
  {S}chr\"odinger system with trapping potentials.
\newblock {\em Discrete Contin. Dyn. Syst.}, 35(12):6085--6112, 2015.

\bibitem{MR3918087}
B.~Noris, H.~Tavares, and G.~Verzini.
\newblock Normalized solutions for nonlinear {S}chr\"{o}dinger systems on
  bounded domains.
\newblock {\em Nonlinearity}, 32(3):1044--1072, 2019.

\bibitem{MR3689156}
D.~Pierotti and G.~Verzini.
\newblock Normalized bound states for the nonlinear {S}chr\"odinger equation in
  bounded domains.
\newblock {\em Calc. Var. Partial Differential Equations}, 56(5):Art. 133, 27,
  2017.

\bibitem{MR3456809}
E.~Serra and L.~Tentarelli.
\newblock Bound states of the {NLS} equation on metric graphs with localized
  nonlinearities.
\newblock {\em J. Differential Equations}, 260(7):5627--5644, 2016.

\bibitem{Nic1}
N.~Soave.
\newblock Normalized ground states for the {NLS} equation with combined
  nonlinearities.
\newblock {\em Preprint arXiv:1811.00826}, 2018.

\bibitem{Nic2}
N.~Soave.
\newblock Normalized ground states for the {NLS} equation with combined
  nonlinearities: the {S}obolev critical case.
\newblock {\em Preprint arXiv:1901.02003}, 2019.

\end{thebibliography}

%

\medskip
\small
\begin{flushright}
\noindent \verb"dario.pierotti@polimi.it"\\
\noindent \verb"nicola.soave@polimi.it"\\
\noindent \verb"gianmaria.verzini@polimi.it"\\
Dipartimento di Matematica, Politecnico di Milano\\
piazza Leonardo da Vinci 32, 20133 Milano (Italy)
\end{flushright}

\end{document}